\documentclass[10pt]{amsart}
\allowdisplaybreaks

\usepackage{amsfonts}
\usepackage{amsmath}
\usepackage{amssymb}
\usepackage{amsthm}
\usepackage{graphicx}
\usepackage{enumerate}
\usepackage{multicol}
\usepackage{mathrsfs}
\usepackage{array}
\usepackage[all,cmtip]{xy}
\usepackage{enumerate}
\usepackage{xcolor}

\newcounter{dummy} \numberwithin{dummy}{section}
\newtheorem{theorem}[dummy]{Theorem}
\newtheorem{corollary}[dummy]{Corollary}
\newtheorem{lemma}[dummy]{Lemma}
\newtheorem{definition}[dummy]{Definition}
\newtheorem{proposition}[dummy]{Proposition}
\theoremstyle{remark}
\newtheorem{remark}[dummy]{Remark}
\newtheorem{example}[dummy]{Example}

\newcommand{\real}{\mathbb{R}}

\newcommand{\calA}{\mathcal{A}}
\newcommand{\calE}{\mathcal{E}}
\newcommand{\calH}{\mathcal{H}}
\newcommand{\calL}{\mathcal{L}}
\newcommand{\calR}{\mathcal{R}}
\newcommand{\calV}{\mathcal{V}}

\newcommand{\frakg}{\mathfrak{g}}
\newcommand{\frakk}{\mathfrak{k}}
\newcommand{\frakp}{\mathfrak{p}}

\DeclareMathOperator{\spn}{span}
\DeclareMathOperator{\pr}{pr}

\DeclareMathOperator{\inc}{inc}
\DeclareMathOperator{\id}{id}
\DeclareMathOperator{\Ann}{Ann}

\DeclareMathOperator{\vl}{vl}
\DeclareMathOperator{\tr}{tr}

\DeclareMathOperator{\Isom}{Isom}
\newcommand{\tensorg}{\mathbf{g}}
\newcommand{\tensorh}{\mathbf{h}}
\newcommand{\tensorr}{\mathbf{r}}
\newcommand{\metricd}{\mathsf{d}}

\DeclareMathOperator{\SO}{SO}

\DeclareMathOperator{\Ad}{Ad}

\DeclareMathOperator{\so}{\mathfrak{so}}
\DeclareMathOperator{\se}{\mathfrak{se}}

\newcommand{\bnabla}{\blacktriangledown}

\DeclareMathOperator{\proj}{\boldsymbol \chi}
\newcommand{\zero}{\mathbf 0}
\newcommand{\ve}{\varepsilon}
\newcommand{\FSO}{{\rm F^{\rm SO}}}
\newcommand{\liftmap}{S}
\newcommand\newbullet{{\kern.8pt\displaystyle\centerdot\kern.8pt}}

\numberwithin{equation}{section}

\title[Submersions, Hamiltonians and rolling manifolds]{Submersions, Hamiltonian systems and optimal solutions to the rolling manifolds problem}
\author[E. Grong]{Erlend Grong}

\address{University of Luxembourg, FSTC, Mathematics Research Unit, 6, rue Coudenhove-Kalergi, L-1359 Luxembourg-Kirchberg, Luxembourg}
\email{erlend.grong@gmail.com}

\subjclass[2010]{37J60, 70E18, 53A17, 49B10}

\keywords{Submersions, Hamiltonian systems, Optimal control, Sub-Riemannian manifolds, Rolling manifolds}

\begin{document}

\let\thefootnote\relax\footnotetext{This research was supported by the Fonds National de la
  Recherche Luxembourg (AFR 4736116 and OPEN Project GEOMREV).}

\begin{abstract}
Given a submersion $\pi:Q \to M$ with an Ehresmann connection~$\calH$, we describe how to solve Hamiltonian systems on $M$ by lifting our problem to $Q$. Furthermore, we show that all solutions of these lifted Hamiltonian systems can be described using the original Hamiltonian vector field on $M$ along with a generalization of the magnetic force. This generalized force is described using the curvature of $\calH$ along with a new form of parallel transport of covectors vanishing on $\calH$. Using the Pontryagin maximum principle, we apply this theory to optimal control problems $M$ and $Q$ to get results on normal and abnormal extremals. We give a demonstration of our theory by considering the optimal control problem of one Riemannian manifold rolling on another without twisting or slipping along curves of minimal length. \end{abstract}

\maketitle

\section{Introduction}
Finding a path of minimal length for rolling a ball on a table without twisting or slipping is a problem with surprising depth and relations to geometry. This was first addressed with the paper of V.~Jurdjevic \cite{Jur93}, where it was shown that the mimimal paths of the plate-ball problem are given by \emph{elastica}. In this setting, an elastic refers to a curve~$\gamma(t)$ in the euclidean space that minimizes the integral over its geodesic curvature, given initial and final values for both position and velocity. A similar result holds in the setting of Riemannian manifolds with constant sectional curvature \cite{Zim05,JuZi08}. We would like to consider a more general case of two arbitrary Riemannian manifolds rolling on each other without slipping or twisting. These restrictions describe situations with high friction, such as two rubber surfaces rolling against each other~\cite{KoEh07}. Rolling of higher dimensional manifolds was first introduced in~\cite{Nom78}. For higher dimensional applications, we mention \cite{SHL06}, where an interpolation problem of a satellite is solved using a rolling of $\SO(3)$ on three-dimensional euclidean space. Several results concerning controllability of this problem exist, see e.g. \cite{BrHs93,ChKo12,CGK13,ChKo12b, CGK15,Gro12}, however, results of optimality have been limited to the case of constant sectional curvature, even for surfaces.

Rather than considering the rolling manifolds problem exclusively, we first want to develop tools to deal with optimal control problems that we will call \emph{lifted}. Locally, we can give the following description of a lifted optimal control problem. Consider an $n+ \nu$-dimensional manifold $Q$ as the configuration space with local coordinates $q = (x, y) = (x_1, \dots, x_n, y_1,\dots, y_\nu)$ and let $M$ be the image of the map~$(x,y) \mapsto x$. Let $U\subseteq \mathbb{R}^k$ be the space of control parameters and consider the control system
$$\mathsf{b}: Q \times U \to TQ, \qquad (q, u) \mapsto \mathsf{b}(q,u).$$
Relative to a fixed $q_0\in Q$, if $[0,T] \to U$, $t \mapsto u(t)$ is a measurable, bounded curve in $U$, write $q_{u}(t) = (x_u(t), y_u(t))$ for the solution of
$$\dot q_{u} (t) = \mathsf{b}(q_u(t),u(t)),  \qquad q_u(0) = q_0.$$
For a given $q_1$, we want to find a control $\tilde u:[0,T] \to U$ such that $q_{\tilde u} (0) = q_0$, $q_{\tilde u}(T) = q_1$ and such that the functional
$$u \mapsto \int_0^T \mathsf{c}(q_u(t), u(t)) dt, \qquad \mathsf{c} \in C^\infty(Q\times U),$$
is minimized. Roughly speaking, we then say that this optimal control problem is lifted from $M$ if $q_u(t)$ is uniquely determined by $x_u(t)$ and $\mathsf{c}$ is constant in the directions of $y_1, \dots, y_\nu$. For a more precise global statement in terms of submersions, see Section~\ref{sec:LiftedOCP}. Rolling a ball on a table can be seen as such an optimal control problem. Our configuration space is five dimensional, with two coordinates for each surface and one coordinate for their relative configuration. However, by the restrictions of high friction, any rolling motion is uniquely determined by its path along the table. Furthermore, it is the length of the path along the table that we are trying to minimize. Similarly, rolling two manifolds against each other without twisting or slipping can be seen as an optimal control problem lifted from one of the manifolds.

We would like to give a general description of solutions to lifted optimal control problems. We will then apply this description to manifolds rolling without twisting or slipping, but the theory is applicable to any optimal control problem satisfying the above description. Since optimal solutions are related to Hamiltonian systems through the Pontryagin maximum principle, we will start with results on Hamiltonian systems on a submersion $\pi: Q \to M$. We introduce lifted Hamiltonian functions in Section~\ref{sec:SubEhresmann}. Using this concept, we will show in Section~\ref{sec:MainTh} that Hamiltonian system on the base space~$M$ can be solved on the top space~$Q$. This possibility can be a great advantage if computations can be done more easily on the top space rather than the base space. Our result is a generalization for the case of Riemannian submersions, where the geodesics of the base space can be found by computing the geodesics on the top space, see e.g. \cite{Her60b}. 
Furthermore, we will show that projections to the base space of solutions of a lifted Hamiltonian system on $Q$ are given by the original Hamiltonian vector field on $M$ in addition to a generalization of the magnetic force. The result of \cite{Mon84} relating principal bundles over Riemannian manifolds with gauge theory is contained as a special case. The key element of our approach is a new form of parallel transport introduced in Section~\ref{sec:Parallel}. We describe how our result can be applied to optimal control problems in Section~\ref{sec:LiftedOCP}. This is applied to sub-Riemannian manifolds in Section~\ref{sec:sR}, giving several results for normal geodesic and abnormal curves.

Understanding how to solve Hamiltonian systems by going ``upstairs'' or ``downstairs'' on a submersion will both be important for finding the equations for the optimal solutions of the rolling problem in Section~\ref{sec:Rolling}. 
We will give the equations for the optimal solution of the rolling manifold problem with some additional details given for the lower dimensional cases. We will also describe the relation of a general $n$-dimensional manifold rolling on a constant curvature space and $\frakg^*$-gauge theory in Section~\ref{sec:SpaceForm}, where $\frakg$ is $\so(n+1)$, $\so(1,n)$ or $\se(n+1)$. In particular, optimal solutions of two rubber surfaces, with one surface being flat, can be considered a physical interpretation of $\se(3)^*$-gauge theory, which has no Riemannian formulation, see Example~\ref{ex:Gauge}. 

Proofs of many of the results in Section~\ref{sec:liftHamilton} are left to Section~\ref{sec:proofs}.

\section{Lifting Hamiltonian systems} \label{sec:liftHamilton}
\subsection{Notation and conventions} All manifolds are smooth and connected. If $p^{\calE}: \calE \to M$ is a vector bundle over a manifold $M$, we will denote the space of its smooth sections by $\Gamma(\calE)$. For any element $X \in \Gamma(\calE)$, we will use $X|_x$ rather than $X(x)$ for its value at $x \in M$. If $\calE$ has a metric tensor $\tensorg$, we will write $\langle \newbullet, \newbullet \rangle_{\tensorg}$ for the fiberwise inner product and $\|e\|_{\tensorg} = \langle e, e \rangle_{\tensorg}^{1/2}$ for the corresponding norm.

For the natural projection from the cotangent bundle $T^*M$, we will simply write $p^M: T^*M \to M$. 
If $X$ is a vector field on $M$ and $\eta$ is a form, we denote their interior product by $\iota_X \eta$ or $\eta(X, \newbullet).$ If $\lambda \in T^*_xM$ is a covector and $v \in T_xM$ is a vector, both over the same point, we write $\lambda(v)$ as just $\lambda v$ whenever there is no possibility of confusion. We will use the same convention for one-forms and vector fields.

\subsection{Lifted Hamiltonian functions and Ehresmann connections} \label{sec:SubEhresmann}
Let $\tilde H:T^* Q \to \mathbb{R}$ be a Hamiltonian function on an $n+\nu$-dimensional manifold $Q$. Assume that locally we can always find coordinates $(x,y) = (x_1, \dots, x_\nu, y_1, \dots, y_\nu)$ and one-forms $\tau_1, \dots, \tau_\nu$ such that if we give a covector $\tilde \lambda$ in $T^*Q$ coordinates $(x,y,a,b)$ whenever
\begin{equation} \label{coordinates}\tilde \lambda = \sum_{i=1}^n a_i dx_i |_{x,y} + \sum_{\kappa=1}^\nu b_\kappa \tau_\kappa|_{x,y},\end{equation}
then we have $\tilde H(\tilde \lambda) = \tilde H(x,y, a,b) = H(x,a)$ for some function $H$. We want to consider such Hamiltonian functions, which we will call \emph{lifted}. We remark that for the coordinates in \eqref{coordinates} to be well defined, $\tau_1, \dots, \tau_\nu$ must be linearly independent and transverse to the span of $dx_1, \dots, dx_n$. We will show that solutions will depend on both the function $H$ and the geometric data of the subbundle $\calH = \cap_{s=1}^\nu \ker \tau_s$ of~$TQ$.

To make our definition precise, let $\pi : Q \to M$ be a surjective submersion between two connected manifolds $Q$ and $M$, that is, a surjective map such that $\pi_*: TQ \to TM$ is surjective as well. Since every point $x \in M$ is a regular value of~$\pi$, $Q_x := \pi^{-1}(x)$ is always an embedded submanifold of~$M$. The tangent bundles of these submanifolds $Q_x$ together form \emph{the vertical bundle} $\calV := \ker \pi_*$ of $\pi$. We will also call elements of $\calV$ \emph{vertical vectors} and say that a vector field is \emph{vertical} if it only takes values in $\calV$.

If $\pi(q) = x$, then we have natural maps $\pi_{*,q}: T_q Q \to T_xM$ and $\pi_q^* :T_x^*M \to T^*_qQ$. To have choices of inverses for these maps, we add the additional structure of \emph{an Ehresmann connection} on $\pi$, that is, a subbundle $\calH$ of $TQ$ such that $TQ= \calH \oplus \calV$. Relative to this Ehresmann connection, define \emph{the horizontal lift} of~$v\in T_x M$ to~$q \in Q_x$ as the unique element $h_qv \in \calH_q$ satisfying $\pi_{*} h_q v = v.$ This gives us a corresponding map $\pi^2:T^*Q \to T^*M$ such that
$$\pi^2(\tilde \lambda )(v) = \tilde \lambda( h_q v), \qquad \tilde \lambda \in T^*_q Q, v \in T_x M, q \in Q_x.$$
In other words, if $\Ann(\calH)$ is the subbundle of covectors vanishing on $\calH$,
$$\Ann(\calH) = \{ \lambda \in T^*_x M \, : \, \lambda v = 0 \text{ for any } v \in \calH_x, x \in M\},$$
then $\pi^2$ is the unique map such that $\ker \pi^2 = \Ann(\calH)$ and such that $\pi^2$ is an inverse of $\pi_q^*$ on every fiber. We use the notation $\pi^2$, since we have the following commutative diagram
$$\xymatrix{T^*Q \ar[r]^{\pi^2} \ar[d]_{p^Q} & T^*M \ar[d]^{p^M} \\ Q \ar[r]_{\pi} & M}$$
For more details on Ehresmann connections, we refer to \cite[Chapter III.9]{KMS93}.
\begin{definition}
We say that a Hamiltonian function~$\tilde H: T^*Q \to \mathbb{R}$ is lifted from~$M$ if~$\tilde H = H \circ \pi^2$ for some Hamiltonian function $H :T^*M \to \mathbb{R}$, with $\pi^2$ defined relative to some Ehresmann connection $\calH$ on $\pi$.
\end{definition}
In the local coordinates of \eqref{coordinates}, we have $\pi(x,y) = x$,
$$\pi^2 :  \sum_{i=1}^n a_i dx_i |_{x,y} + \sum_{\kappa=1}^\nu b_\kappa \tau_\kappa|_{x,y} \mapsto  \sum_{i=1}^n a_i dx_i |_x,$$
$\calV = \spn \{ \partial_{y_1}, \dots, \partial_{y_\nu}\}$, $\Ann(\calV) = \spn\{ dx_1, \dots, dx_n\}$, $\Ann(\calH) = \spn \{\tau_1, \dots, \tau_\nu\}$ and $\calH = \cap_{\kappa=1}^\nu \ker \tau_\kappa.$ Furthermore,
$$h_{x,y} \partial_{x_j} = \partial_{x_j} |_{x,y} - \pr_{\calV} \partial_{x_j} 
|_{x,y} = \partial_{x_j} |_{x,y} - \sum_{\kappa=1}^\nu c_{\kappa j}(x,y) \partial_{y_\kappa} |_{x,y},$$
with $(c_{\kappa j}) = A B^{-1}$, where $A$ is the matrix $A = (A_{\kappa j} ) = (\tau_\kappa(\partial_{x_j}))$ and $B$ is the invertible matrix $B = (B_{\kappa\mu} ) = (\tau_\kappa (\partial_{y_\mu})).$

\subsection{Parallel transport of vertical vectors} \label{sec:Parallel}
If we have a given Ehresmann connection $\calH$ on a submersion $\pi:Q \to M$, we can also define horizontal lift of vector fields and curves. For a vector field $X$ on $M$, we define a vector field $hX$ on $Q$ by $hX |_q = h_q X|_{\pi(q)}$. As for curves, we say that a curve~$\tilde \gamma(t)$ in~$Q$ is \emph{$\calH$-horizontal} if~$\tilde \gamma(t)$ is absolutely continuous and satisfies $\dot{\tilde \gamma}(t) \in \calH_{\tilde \gamma(t)}$ for almost every $t$. We say that $\tilde \gamma: [0,T] \to Q$ is a horizontal lift of the curve $\gamma:[0,T] \to M$ to $q_0 \in Q_{\gamma(0)}$ if $\tilde \gamma(t)$ is $\calH$-horizontal, $\tilde \gamma(0) = q_0$ and $\pi(\tilde \gamma(t)) = \gamma(t).$ In other words, $\tilde \gamma(t)$ is a horizontal lift of $\gamma(t)$ if it is a solution to the initial value problem
\begin{equation} \label{IVP} \dot{ \tilde \gamma}(t) = h_{\tilde \gamma(t)} \dot \gamma(t), \qquad \tilde \gamma(0) = q_0 \in Q_{\gamma(0)}.\end{equation}
The formulation \eqref{IVP} shows that we have uniqueness of horizontal lifts. However, existence is only assured for sufficiently short time, i.e. for a a sufficiently small value of $\ve >0$ there is a horizontal lift of $\gamma|_{[0,\ve]}$. If horizontal lifts of any absolutely continuous curve in $M$ exist for all time, then $\calH$ is called \emph{complete}.

Horizontal lifts of curves can be considered as a generalization of parallel transport with respect to an affine connection. Indeed, let $\nabla$ be an affine connection defined on the tangent bundle $p^{TM} : TM \to M$. We can then define an Ehresmann connection $\calH^{\nabla}$ on $p^{TM}$ by the property that any smooth curve $X(t)$ in $TM$ is $\calH^{\nabla}$-horizontal if and only if $X(t)$ is a parallel vector field along its projection $\gamma(t) = p^{TM} (X(t))$ in $M$. In this case, horizontal lifts of curves in $M$ to $TM$ are just the parallel transport of vectors with respect to $\nabla$.

However, there exists a second point of view regarding parallel transport of affine connections that involves \emph{vertical lifts}. Let $p^{\calE} : \calE \to M$ be a vector bundle. For any pair of elements $e_1, e_2 \in \calE_x, x \in M$ we define the vertical lift of $e_2$ to $e_1$ by
$$\vl_{e_1} e_2 = \left. \tfrac{d}{dt} (e_1 + te_2) \right|_{t=0} \in T_{e_1} \calE.$$
Similarly, for any element $X \in \Gamma(\calE)$, we define the vertical lift of $X$ as the vector field $\vl X\in \Gamma(T\calE)$ given by
\begin{equation*}  \vl X |_e = \vl_e  X|_{p^{\calE}(e)} \text{ for any } e \in \calE.\end{equation*}
Note that vertical lifts always take values in $\ker p_*^{\calE}.$
With this definition, it can be verified that for any pair of vector fields $X,Y \in \Gamma(TM)$, we have the relation
$$[hX, \vl Y] = \vl \nabla_X Y,.$$
where $hX\in \Gamma(T(TM))$ is the horizontal lift to $TM$ with respect to $\calH^{\nabla}$. We would like to generalize this idea of parallel transport to submersions, even if we have no notion of vertical lifts in this setting.

Recall that if $f:M \to \widehat M$ is a smooth map between two manifolds, then two vector fields $X$ and $\widehat X$ on $M$ and $\widehat M$ respectively are called $f$-related if $f_* X |_x = \widehat X|_{f(x)}$ for any $x \in M$. If $X$ and $Y$ are $f$-related to $\widehat X$ and $\widehat Y$, respectively, then $f_* [X,Y]|_x = [ \widehat X, \widehat Y] |_{f(x)}$. Let $\calH$ be an Ehresmann connection on the submersion $\pi: Q \to M$ with vertical bundle $\calV = \ker \pi_*$. We introduce an operator $\bnabla: \Gamma(TM) \times \Gamma(\calV) \to \Gamma(\calV)$ defined by
\begin{equation} \label{bnabla} \bnabla_X V := [hX, V], \qquad X \in \Gamma(TM), V \in \Gamma(\calV).\end{equation}
This map is indeed well defined, since $hX$ and $V$ are $\pi$-related respectively to $X$ and the zero section of $TM$, so their bracket $[hX, V]$ must be $\pi$-related to the zero section as well. It is only $\real$-linear in the second argument, but $C^\infty(M)$-linear in the first argument, making $\bnabla_X V|_q$  only depend on $X|_{\pi(q)}$. Hence, for any $v \in T_x M$ and $V \in \Gamma(\calV)$, $\bnabla_v V |_q$ is well defined for any $q \in Q_x$. For this reason, we can consider $\bnabla_v V$ as a vector field on~$Q_x$ for any $v \in T_xM$ and $V \in \Gamma(\calV)$. Finally, it follows that if $\gamma(t)$ is a curve in $M$ with a horizontal lift $\tilde \gamma(t)$ in $Q$ and if $V(t)$ is a vertical vector field along $\tilde \gamma(t)$, then $\bnabla_{\dot \gamma} V(t)$ is well defined. We will use this fact to introduce parallel transport of vertical vectors along curves in $M$.

\begin{lemma} \label{lemma:translate}
Let $\gamma: [0,T] \to M$ be an absolutely continuous curve with $\gamma(0) = x_0$. Let $q_0$ be an element in $Q_{x_0}$ and assume that the horizontal lift $\tilde \gamma(t)$ of $\gamma(t)$ with $\tilde \gamma(0) = q_0$ exists. Then, for any $V_0 \in \calV_{q_0}$ there is a unique vertical vector field $V(t)$ along $\tilde \gamma(t)$ which solves the equation
$$\bnabla_{\dot \gamma} V(t) = 0 \quad \text{ and }\quad V(0) = V_0.$$
\end{lemma}

\begin{proof}
Let the respective dimensions of $M$ and $Q$ be $n$ and $n + \nu$. Around the point $q_0$, choose a coordinate system $(x,y) = (x_1, \dots, x_n, y_1, \dots, y_\nu)$ such that the submersion $\pi$ can be written as $\pi(x,y) = x$. Since the image of $\tilde \gamma$ is compact, we may assume that $\tilde \gamma$ is contained in the domain of the local coordinate system by taking a finite subdivision. Write
$$\dot \gamma(t) =  \sum_{i=1}^n \dot x_i(t) \partial_{x_i} \quad \text{and } \quad \bnabla_{\partial_{x_i} } \partial_{y_\kappa} = \sum_{\mu=1}^\nu {\bf \Gamma}_{i\kappa}^\mu \partial_{y_\mu},$$
Then $\dot {\tilde \gamma}(t) = \sum_{i=1}^n \dot x_i(t) \, h\tfrac{\partial}{\partial x_i}$ and $V(t) = \sum_{\kappa=1}^\nu v_\kappa(t) V |_{\tilde \gamma(t)}$ is a solution to
$$\bnabla_{\dot \gamma} V(t) = \sum_{\kappa =1}^\nu \left(\dot v_\kappa(t) + \sum_{i=1}^n \sum_{\mu=1}^\nu x_i(t) v_\mu(t) {\bf \Gamma}^\kappa_{i\mu} \right) \partial_{y_\nu} |_{\tilde \gamma(t)}= 0.$$
These equations, along with the initial conditions, uniquely determine $V$ and since it is given by a linear system, we also have existence.
\end{proof}

From Lemma~\ref{lemma:translate}, given an element $q_0 \in Q_{x_0}$ and a curve $\gamma:[0,T] \to M$ with $\gamma(0) = x_0$, we can define parallel transport of any element $V_0 \in \calV_{q_0}$. However, this parallel transport is only well defined for sufficiently small times such that the horizontal lift of $\gamma$ exists. If $\calH$ is a complete Ehresmann connection, then parallel transport of elements in $\calV$ along curves in $M$ is defined for all time.

Like elements in $\calV$, we can define parallel transport of elements in $\Ann(\calH)$ by $\bnabla_X \beta := \pr_{\calV}^* \calL_{hX} \beta$ for any $X \in \Gamma(TM)$ and $\beta \in \Gamma(\Ann(\calH))$. Here $\calL_{hX}$ denotes the Lie derivative and $\pr_{\calV}$ is the projection to $\calV$ with respect to the decomposition $TQ = \calH \oplus \calV$ with pull-back $\pr_{\calV}^*$, i.e.,  $(\pr_{\calV}^* \lambda)(v) = \lambda (\pr_{\calV} v)$. Note that if $X \in \Gamma(TM)$, $\beta \in \Gamma(\Ann(\calH))$ and $V \in \Gamma(\calV)$, then
$$hX ( \beta V) = (\bnabla_X \beta)V + \beta(\bnabla_X V).$$

\subsection{Curvature and lifted Hamiltonian functions} \label{sec:MainTh} Let $M$ be a manifold with a Hamiltonian function $H \in C^\infty(T^*M)$. Let $\pi: Q \to M$ be a surjective submersion with Ehresmann connection $\calH$, and let $\tilde H = H \circ \pi^2$ be the corresponding lifted Hamiltonian. Associated to every Ehresmann connection $\calH$, we have a vector-valued two-form $\calR \in \Gamma(\bigwedge^2 T^*Q \otimes TQ)$ called \emph{the curvature of} $\calH$. For any pair of vector fields $\tilde X, \tilde Y \in \Gamma(TQ)$, define
$$\calR(\tilde X, \tilde Y) = \pr_{\calV} [\pr_{\calH} \tilde X, \pr_{\calH} \tilde Y],$$
where $\pr_{\calV}$ and $\pr_{\calH}$ are the respective projections to $\calV$ and $\calH$ with respect to the decomposition $TQ = \calH \oplus \calV$. Since this map is $C^\infty(Q)$-linear in each coordinate, $\calR$ is a well defined vector-valued two-form.

We remark that for any pair of vector fields $X, Y \in \Gamma(TM)$, we have $[hX, hY] = h[X,Y] + \calR(hX, hY)$ since $hX$ and $hY$ are $\pi$-related to $X$ and $Y$, respectively. With slight abuse of notation, we will write $\calR(hX, hY)$ simply as $\calR(X,Y)$. Similarly, for any pair of vectors $v,w \in T_xM$, we let $\calR(v,w)$ denote the vector field on $Q_x = \pi^{-1}(x)$ given by $q \mapsto \calR(h_q v, h_q w).$ We also note that by the Jacobi identity,
$$\left( \bnabla_X \bnabla_Y - \bnabla_Y \bnabla_X - \bnabla_{[X,Y]} \right)V = \left[ \calR(X,Y), V \right], \quad X,Y \in \Gamma(TM), V \in \Gamma(\calV).$$

With this formalism in place, we are ready to state our main theorem. For any Hamiltonian function $H$, we let $\vec{H}$ be its Hamiltonian vector field. We will use the term \emph{integral curve of} $H$ to mean an integral curve of its Hamiltonian vector field. Note that our decomposition of the tangent bundle $TQ = \calH \oplus \calV$ gives us a corresponding direct sum representation $T^*Q = \Ann(\calV) \oplus \Ann(\calH)$ of the cotangent bundle.

\begin{theorem} \label{th:main}
Let $H\in C^\infty(T^*M)$ be a Hamiltonian function on $M$ and define $\tilde H = H \circ \pi^2$, where $\pi^2$ is induced by some Ehresmann connection $\calH$ on $\pi:Q \to M$ with curvature $\calR$.
\begin{enumerate}[\rm (a)]
\item A curve in $\lambda(t)$ is an integral curve of $H$ if and only if every sufficiently short segment is the projection of an integral curve $\tilde \lambda(t)$ of $\tilde H$ which is contained in~$\Ann(\calV)$. Actually, it is sufficient to require that $\tilde \lambda(t)$ meets $\Ann(\calV)$ in just one point.
\item Let $\tilde \lambda: [0,T] \to T^*Q$ be an integral curve of $\tilde H$ with
$$\begin{array}{m{3cm}m{3cm}} $\beta_0 := \pr_{\calV}^* \tilde \lambda(0)$ & \xymatrix{\tilde \lambda(t) \ar@{|->}[r]^{\pi^2} \ar@{|->}[d]_{p^Q} & \lambda(t) \ar@{|->}[d]^{p^M} \\
\tilde \gamma(t) \ar@{|->}[r]_{\pi} & \gamma(t) } \end{array}.$$
The curve $\tilde \gamma(t)$ is then an $\calH$-horizontal lift of $\gamma(t)$. The curve $\lambda(t)$ is a solution to the equation
\begin{equation} \label{maineq1} \dot \lambda(t) = \vec{H} |_{\lambda(t)} + \vl_{\lambda(t)} \beta(t) \calR(\dot \gamma(t), \newbullet ) \, ,\end{equation}
where $\beta(t)$ is the curve in $\Ann(\calH)$ defined as the parallel transport of~$\beta_0$ along~$\gamma(t)$, that is
\begin{equation} \label{maineq2} \bnabla_{\dot \gamma} \beta (t) = 0 \qquad \beta(0) = \beta_0.\end{equation}
\end{enumerate}
\end{theorem}

The above results can be interpreted in the following way. First of all, Theorem~\ref{th:main}~(a) tells us that given a Hamiltonian system on $M$, we can solve this Hamiltonian system by only doing our computations on $Q$. Even though $Q$ is a larger space, there might be reasons why computations on $Q$ are simpler, e.g. existence of group actions or a trivial tangent bundle. An example where we have both of these benefits are homogeneous spaces $M = G/K$, where $G$ is a Lie group and $K \subseteq G$ is a closed subgroup. Any choice of Ehresmann connection on $\pi: G \to M$ will allow us to solve Hamiltonian systems on $M$ by doing computations in $G$. 
The result of Theorem~\ref{th:main}~(b) tells us that as long as we understand the behavior of $\bnabla$-parallel transport, we can solve some Hamiltonian systems of $Q$ on $M$ as well. By the proof of Theorem~\ref{th:main}, it follows that $\beta(t) = \pr_{\calV}^* \tilde \lambda(t)$, hence $\tilde \lambda(t)$ is given as $\tilde \lambda(t) = \pi^*_{\tilde \gamma(t)} \lambda(t) + \beta(t).$ We leave the proof of Theorem~\ref{th:main} to Section~\ref{sec:proofmain}.

\subsection{Special cases}
We look at two important special cases; one where the Hamiltonian $H$ comes from a Riemannian metric and one where the submersion $\pi$ is a principal bundle. We end with the case of charged particles in a Riemannian manifold, which is included in both of these special cases.

\subsubsection{Hamiltonian of a Riemannian manifold} \label{sec:Riemann} Let $(M, \tensorg)$ be a Riemannian manifold where $\tensorg$ is the metric tensor. Let $\flat: TM \to T^*M$ be the map $v \mapsto \langle v, \newbullet \rangle_{\tensorg}$ with inverse $\sharp$. Associated to the Riemannian structure, we have a Hamiltonian function $H(\lambda) = \frac{1}{2} \lambda(\sharp \lambda).$ The projections of integral curves of $H$ give us Riemannian geodesics, i.e., solutions to the equation $\nabla_{\dot \gamma} \dot \gamma(t) = 0$ where $\nabla$ is the Levi-Civita connection of $\tensorg$. 

\begin{corollary} \label{cor:Riemann}
Let $\pi: Q \to M$ be a surjective submersion into a Riemannian manifold $(M, \tensorg)$. Let $H$ be the Hamiltonian associated to $\tensorg$ and let $\calH$ be a chosen Ehresmann connection with corresponding map $\pi^2: T^*Q \to T^*M$.

Let $\tilde \lambda:[0,T] \to T^*Q$ be any integral curve of $H \circ \pi^2$ with projection $\tilde \gamma(t)$ in $Q$ and with $\beta_0 = \pr_{\calV}^* \tilde \lambda(0)$. Then $\tilde \gamma(t)$ is an $\calH$-horizontal lift of $\gamma(t)$, the latter being a solution of the equation
$$\nabla_{\dot \gamma} \dot \gamma(t) = \sharp \beta(t) \calR(\dot \gamma(t), \newbullet), \qquad \bnabla_{\dot \gamma} \beta(t) = 0, \qquad \beta(0) = \beta_0.$$
\end{corollary}
\begin{proof}
Pick any local orthonormal basis $X_1, \dots, X_n$ of vector fields relative to~$\tensorg$ and use these vector fields to give the fibers of cotangent bundle $T^*M$ the coordinates $\lambda_i = \lambda X_i$. It is a standard result that the Hamiltonian vector field of $H$ in these coordinates is given by
$$\vec{H} = \sum_{i=1}^n \lambda_i X_i - \sum_{i,j,k=1}^n \lambda_j \lambda_k \Gamma_{jk}^i \partial_{\lambda_i} , \qquad \Gamma^k_{ij} := \left\langle X_k, \nabla_{X_i} X_j \right\rangle_{\tensorg}.$$

Write $\lambda(t) = \pi^2(\tilde \lambda(t)) = \sum_{i=1}^n \lambda_i(t) \flat X_i |_{\gamma(t)}$. Then from Theorem~\ref{th:main}~(b), we obtain
\begin{align*} \dot \lambda_i(t) & = d\lambda_{i}( \dot \lambda(t)) = d\lambda_i \left(\vec{H}|_{\lambda(t)} + \vl\left(\lambda(t), \beta(t) \calR(\dot \gamma(t), \newbullet) \right) \right) \\
 &=  - \sum_{i,j,k=1}^n \lambda_j(t) \lambda_k(t) \Gamma_{jk}^i + \beta(t) \calR(\dot \gamma(t), X_i).
\end{align*}
Hence, since $\dot \gamma(t) = \sum_{i=1}^n \left(\flat X_i(\dot \lambda(t))\right) X_i = \sum_{i=1}^n \lambda_i(t) X_i$, we obtain
\begin{align*} \nabla_{\dot \gamma} \dot \gamma(t) = & \sum_{i=1}^n \dot \lambda(t)  X_i |_{\gamma(t)} + \sum_{i,j,k=1}^n \lambda_i(t) \lambda_j(t) \Gamma^k_{ij} X_k |_{\gamma(t)} \\
= & \sum_{i=1}^n \left( \beta(t) \calR(\dot \gamma(t), X_i) \right) X_i = \sharp \beta(t) \calR(\dot \gamma(t), \newbullet).
\end{align*}
\end{proof}

\subsubsection{Principal connections on principal bundles} \label{sec:Principal} Let us consider the case when $\pi: Q \to M$ is a principal $G$-bundle, with $G$ acting on the right. An Ehresmann connection $\calH$ on $\pi$ is itself called \emph{principal} if $\calH_q \cdot a = \calH_{q \cdot a}$ for every $q \in Q$ and $a \in G$. Principal Ehresmann connections are always complete. Every such Ehresmann connection is uniquely determined by its corresponding \emph{connection} \emph{form}~$\omega$, which is a one-form taking its values in the Lie algebra $\frakg$ of $G$. For any $\frakg$-valued function $f \in C^\infty(Q,\frakg)$, define a vector field $\xi(f)$ on $Q$ by
\begin{equation} \label{sigmaVF} \xi(f)|_q = \left. \frac{d}{dt} q \cdot \exp_G(t f(q)) \right|_{t=0}.\end{equation}
In particular, this gives us a vector field $\xi(A)$ for any constant element $A \in \frakg$.
A~connection form on $\pi$ is then a $\frakg$-valued one-form $\omega$ satisfying
$$\omega(\tilde v \cdot a) = \Ad(a^{-1}) \omega(\tilde v) \quad \text{and} \quad \omega(\xi(A)) = A, \qquad \tilde v \in TQ, a \in G, A \in \frakg.$$
Any principal Ehresmann connection $\calH$ uniquely determines a connection form $\omega$ by $\ker \omega = \calH$ and $\omega(\xi(A)) = A.$
We can also define \emph{the curvature form} $\Omega$ by
$$\Omega(\tilde X, \tilde Y) = d\omega(\tilde X, \tilde Y) + \omega([\tilde X, \tilde Y]) = - \omega(\calR(\tilde X, \tilde Y)).$$

Introduce the vector bundle $\Ad(Q) \to M$ by $\Ad(Q) = (Q \times \frakg)/G$ where the right action of $G$ on $Q \times \frakg$ is given by
$$(q, A) \cdot a = (q \cdot a, \Ad(a^{-1}) A), \qquad q \in Q, A \in \frakg, a \in G.$$
Any section $s$ of $\Ad(Q)$ can equivalently be considered as a function $F^s: Q \to \frakg$ satisfying $F^s(q \cdot a) = \Ad(a^{-1}) F^s(q),$ that is, a $G$\emph{-equivariant function}. This function is defined such that if $s|_{\pi(q)} = (q, A)/G,$ then $F^s(q) = A$.

For any $G$-equivariant function $F^s$ and any vector field $X$ on $M$, the function $dF^s(hX)$ is also $G$-equivariant and so can be considered as a section of~$\Ad(Q)$ as well. We will denote this section by $\nabla_X^\omega s$, giving us an affine connection $\nabla^\omega$ on~$\Ad(Q)$. The property $\Omega(\tilde X \cdot a, \tilde Y \cdot a) = \Ad(a^{-1}) \Omega(\tilde X, \tilde Y)$ means that we can consider $\Omega$ as a two-form on $M$ with values in $\Ad(Q)$. The connection $\nabla^\omega$ induces a connection on the dual $\Ad(Q)^*$ which we denote by the same symbol. We identify $\Ad(Q)^*$ with $\Ad^*(Q)$, where $\Ad^*(Q)$ is defined as $Q \times \frakg^*$ divided out by the action $(q, A^*) \mapsto (q \cdot a, \Ad^*(a^{-1}) A^*)$, $a \in G$, $A^* \in \frakg^*$.

\begin{corollary} \label{cor:Principal}
Let $H$ be a Hamiltonian function on $T^*M$ and let $\tilde H = H \circ \pi^2$ be a Hamiltonian function on $T^*Q$, where $\pi^2: T^* Q \to T^*M$ is the map corresponding to $\calH$. Let $\tilde \lambda:[0,T] \to Q$ be an integral curve of $\tilde H$ and let $\lambda(t) = \pi^2(\tilde \lambda(t))$ be its image in $T^*M$. Assume that $\tilde \lambda(0) \in T_{q_0}Q$ and that $A^* = \tilde \lambda(0) \xi(\newbullet) \in \frakg^*.$

Then $\lambda(t)$ is a solution to
$$\dot \lambda(t) = \vec{H}|_{\lambda(t)} - \vl_{\lambda(t)} c(t) \Omega( \dot \gamma(t), \newbullet),$$
where $c(t)$ is a curve in $\Ad^*(Q)$ over $\gamma$, uniquely determined by
$$\nabla_{\dot \gamma}^\omega c(t) =0, \qquad c(0) = (q_0, A^*)/G.$$
\end{corollary}

\begin{proof}
Notice first that for any $f \in C^\infty(Q, \frakg)$, we have
$$\bnabla_X \xi(f) = [hX, \xi(f)] = \xi(hX f).$$
It follows that $\bnabla_{\dot \gamma} V(t) = 0$ for some curve $\gamma(t)$ in $M$ if and only if $V(t) = \xi(A)|_{\tilde \gamma(t)}$ for any $A \in \frakg$ and some $\calH$-horizontal lift $\tilde \gamma(t)$ of $\gamma(t)$.

On the other hand, if $s \in \Gamma(\Ad(Q))$ with corresponding equivariant function $F^s$, then for any smooth curve $\gamma(t)$, we have $\nabla_{\dot \gamma}^\omega s(t) = 0$ if and only if $\frac{d}{dt} F^s(\tilde \gamma(t)) = 0$ for any horizontal lift $\tilde \gamma(t)$. It follows that if $s(t)$ is any curve in $\Ad(Q)$ over $\gamma(t)$ such that $\nabla_{\dot \gamma}^\omega s(t) = 0$, then $s(t) = (\tilde \gamma(t), A) /G$ for some constant $A \in \Ad(Q)$ and horizontal lift~$\tilde \gamma(t)$.

Putting these two facts together, if $V(t)$ is a vertical vector field along a horizontal lift $\tilde \gamma(t)$ in $Q$ of $\gamma(t)$, then
$$z(\bnabla_{\dot \gamma} V(t)) = \nabla^\omega_{\dot \gamma} z(V(t)),$$
$$\text{where } \qquad z: \calV \to \Ad(Q), \qquad z: \tilde v \in \calV_q \mapsto (q, \omega(\tilde v))/G.$$
The result now follows from Theorem~\ref{th:main}~(b).
\end{proof}

\begin{example} \label{ex:Gauge}
Let $(M, \tensorg)$ be a Riemannian manifold with corresponding Hamiltonian function $H = \frac{1}{2} \lambda(\sharp \lambda)$. Let $\pi: Q \to M$ be a principal $G$-bundle with a principal Ehresmann connection~$\calH$. Let $\omega$ and $\Omega$ be respectively the connection form and curvature form of~$\calH$ with values in $\frakg$. Let $\pi^2: T^*Q \to T^*M$ also correspond to~$\calH$. Let $\tilde \lambda(t)$ be an integral curve of $H \circ \pi^2$ in $T^*Q$ and define $\gamma(t) = \pi(p^Q(\tilde \lambda(t)))$. Combining the results of Corollary~\ref{cor:Riemann} and Corollary~\ref{cor:Principal}, we get that $\gamma(t)$ is a solution to
\begin{equation} \label{Gauge} \nabla_{\dot \gamma} \dot \gamma(t) = - \sharp c(t) \Omega(\dot \gamma(t), \newbullet), \qquad \nabla_{\dot \gamma}^\omega c(t) = 0.\end{equation}
These are the equations of a free particle in $M$ with a ``color charge'' or gauge $c$ in a Yang-Mills field $\Omega$. If $G$ is $U(1)$ or $\real$, then $c$ represents the charge and $\Omega$ represents a magnetic field. If $\frakg$ has a bi-invariant metric $\langle \newbullet, \newbullet \rangle_{\frakg}$, then the curves in \eqref{Gauge} also appear as projections of the geodesics of the Riemannian metric $\tilde \tensorg$ on $Q$ given by $\langle \tilde v, \tilde w \rangle_{\tilde \tensorg} = \langle \pi_* \tilde v, \pi_* \tilde w \rangle_{ \tensorg} + \langle \omega(\tilde v), \omega(\tilde w) \rangle_{\frakg}.$ In physics, the most important cases of gauge theory are $\frakg = \mathfrak{u}(1), \mathfrak{su}(2)$ and $\mathfrak{su}(3)$, which all have bi-invariant metrics. However, this is not the case for $\frakg = \se(n)$ considered in Section~\ref{sec:SpaceForm}. For more details, see \cite{Mon84} and \cite[Chapter 12]{Mon02}.
\end{example}

\subsection{Lifted optimal control problems} \label{sec:LiftedOCP}
There are many different definitions and generalizations of an optimal control problem. We will use the definition found in \cite[Section 2.1]{AgGa97} and \cite{Agr08}. \emph{A smooth control system} consists of a fiber bundle ${\sf a}: \calA \to M$ with fiber $U$ along with a bundle morphism
$$\xymatrix{\calA \ar[rr]^{\sf b} \ar[dr]_{\sf a} & &  TM \ar[ld]^{p^{TM}} \\ & M}.$$
A curve $A(t)$ in $\calA$ is called {\it an admissible control} if its projection $\gamma(t)$ in $M$ is an $L^\infty$-curve satisfying $\dot \gamma(t) = {\sf b} (A(t))$. We want to consider the following optimal control problem: For a smooth function ${\sf c}:\calA \to \real$ and two points $x_0, x_1 \in M$, find the admissible control $A:[0,T] \to \calA$ which satisfies
$${\sf a}(A(0)) = x_0, \qquad {\sf a}(A(T)) = x_1,$$
and minimize the functional
$$A \mapsto \int_0^{T} {\sf c}(A(t)) \, dt.$$
The latter functional is called {\it the cost functional} and ${\sf c}$ is called \emph{the cost function}.
\begin{definition}
We will denote the above optimal control problem as the optimal control problem associated to $({\sf a}, {\sf b}, {\sf c})$.
\end{definition}

A sufficient condition for an admissible control to be a solution to this optimal control problem is given by the Pontryagin maximum principle (PMP). In order to present this result in a simpler way, we will write the formulation assuming that $\calA$ is a trivial fiber bundle $\calA = M \times U$, which can be considered a local version of the general case. See Remark \ref{re:PMPnontriv} for how this statement can be reformulated for the case when $\calA$ cannot be trivialized. For the proof of this theorem, we refer to \cite[Theorem~12.3]{AgSa04}.

\begin{theorem} \label{th:PMP} {\bf PMP for Optimal Control Problem with fixed time~$T$}  \\
Let $\bar{A}(t) = (\gamma(t), \bar{u}(t))$ be a solution to the optimal control problem associated to $({\sf a}, {\sf b}, {\sf c})$, where $\gamma(t)$ is a curve in $M$ and $\bar{u}(t)$ is a curve in $U$, both with domain~$[0,T]$. For each $\nu \in \real, u \in U$, consider a Hamiltonian function
\begin{equation} \label{PMPHamiltonian} H_{\nu,u}(\lambda) =\lambda \, {\sf b}(x,u) + \nu {\sf c}(x,u), \qquad \lambda \in T_x^*M.\end{equation}
Then there exists a curve $[0,T] \to T^*M$, $t \mapsto \lambda(t)$ and a number $\nu \leq 0$ such that
\begin{enumerate}[\rm (i)]
\item $p^{M} (\lambda(t)) = \gamma(t).$
\item $\dot \lambda (t) = \vec{H}_{\nu, \bar{u}(t)} |_{\lambda(t)}$ for almost every $t$,
\item $H_{\nu,\bar{u}(t)}(\lambda(t)) = \max\limits_{u \in U} H_{\nu,u} (\lambda(t))$  for almost every $t$.
\end{enumerate}
Moreover, $\lambda$ never intersects the zero section $\zero$ of $T^*M$ if $\nu = 0$.
\end{theorem}
We will use the name {\it extremals} for solutions of PMP. They are called {\it normal} if $\nu < 0$ (it is sufficient to consider $\nu=-1$) and {\it abnormal} if $\nu=0$. As is seen in the definition of $H_{\nu,u}$, abnormal extremals do not depend on the function ${\sf c}$, only the control system~$({\sf a}, {\sf b})$.

\begin{remark} \label{re:PMPnontriv}
If the fiber bundle cannot be trivialized, the PMP can be reformulated in the following way. For any $\nu \in \real$, define ${\sf H}_\nu: \calA \times_M T^*M \to \real$ by ${\sf H}_\nu(A,\lambda) = \lambda\, {\sf b}(A) + \nu {\sf c}(A)$, which takes the place of the Hamiltonian in~\eqref{PMPHamiltonian}. Requirement~(ii) is then replaced with the identity
$$\sigma(\dot \lambda(t), \pr_2 |_{\bar{A}(t), \lambda(t)} \newbullet) = d{\sf H}_\nu|_{\bar{A}(t), \lambda(t)},$$
where $\pr_2: \calA \times_M T^*M \to T^*M$ is the projection and $\sigma$ is the canonical symplectic form on $T^*M$. In requirement (iii), the maximum needs to hold over all elements in~$\calA_{\gamma(t)}$.
\end{remark}
We will call ${\sf H}_\nu$ {\it the PMP-Hamiltonian} of the optimal control problem associated to~$({\sf a}, {\sf b} , {\sf c})$. Consider any optimal control problem on $M$ associated to some triple~$({\sf a}, {\sf b} , {\sf c})$. Let $\pi:Q \to M$ be a submersion into $M$ with Ehresmann connection $\calH$ and let $\pi^2: T^*Q \to T^*M$ be defined relative to $\calH$. We then have a lifted optimal problem on $Q$ associated to a triple $(\tilde {\sf a}, \tilde {\sf b}, \tilde {\sf c}).$ This triple is defined by the points (a), (b)  and (c) below.
\begin{enumerate}[\rm (a)]
\item Define the fiber bundle $\tilde {\sf a}: \pi^*\calA \to Q$ as the pull-back bundle of ${\sf a}: \calA \to M$. Recall that this fiber bundle is defined as
$$\pi^* \calA = \left\{ (q,A) \in Q \times \calA \, \colon \, \pi(q) = {\sf a}(A) \right\}.$$

\item Define a bundle morphism $\tilde {\sf b}: \pi^*\calA \to TQ$ by
$$\tilde {\sf b}(q,A) = h_q {\sf b}(A) \qquad \text{for any } (q,A) \in \pi^*\calA,$$
with $h_q$ being the horizontal lift with respect to $\calH$.
\item Let $\pr_{\calA}: \pi^* \calA \to \calA$ be the map $(q, A) \mapsto A$. Then $\tilde {\sf c} \in C^\infty(\pi^*\calA)$ is defined by $\tilde {\sf c} = {\sf c} \circ \pr_{\calA}.$
\end{enumerate}
For this system, we have the following result.

\begin{proposition}
For $\nu \leq 0$, let ${\sf H}_\nu: \calA \times_M T^*M \to \real$ be the PMP-Hamiltonian of the optimal control problem associated to $({\sf a}, {\sf b}, {\sf c})$. Then ${\sf H}_\nu \circ (\pr_{\calA} \times_Q \pi^2)$ is the PMP-Hamiltonian of $(\tilde {\sf a}, \tilde {\sf b}, \tilde {\sf c}).$ As a consequence, the following holds.
\begin{enumerate}[\rm (a)]
\item A curve $\lambda(t)$ in $T^*M$ is a normal (resp. abnormal) extremal if and only if $\lambda(t)$, at least for short time, is the projection of a normal (resp. abnormal) extremal in $T^*Q$ contained in $\Ann(\calV)$ (it is sufficient to require this in only one point).
\item If $\tilde A(t)$ is a solution to the optimal control problem associated to $(\tilde {\sf a}, \tilde {\sf b}, \tilde {\sf c})$, then there is a number $\nu \leq 0$ and a curve $\lambda(t)$ in $T^*M$ (in $T^*M \setminus \zero$ if $\nu=0$) such that
\begin{enumerate}[\rm (i)]
\item $p^{M}(\lambda(t)) =  {\sf a}(\pr_{\calA} \tilde A(t)) =: \gamma(t)$ and $\tilde \gamma(t) = \tilde{\sf a}( \tilde A(t))$ is a horizontal lift of~$\gamma(t)$,
\item of $\pr_{\calA} \tilde A(t) = A(t)$, then for almost every $t$,
$$\sigma(\dot \lambda(t), \pr_2 |_{A(t), \lambda(t)} \newbullet) = d{\sf H}_\nu|_{A(t), \lambda(t)} - \beta(t) \calR(\dot \gamma(t), p^{M}_* \pr_2 |_{A(t), \lambda(t)} \newbullet).$$
with $\beta(t)$ being a form along $\tilde \gamma(t)$, vanishing on $\calH$ and satisfying
$$\bnabla_{\dot \gamma} \beta(t) = 0.$$
\item for almost every $t$,
$${\sf H}_\nu(A(t), \lambda(t)) = \max_{k \in \calA_{\gamma(t)}} {\sf H}_\nu(k,\lambda(t)).$$
\end{enumerate}
\end{enumerate}
\end{proposition}
\begin{proof}
We only need to show the result locally, so we may assume that $\calA = M \times U$ and consequently $\pi^* \calA = Q \times U$.
We then verify that for any $\tilde \lambda \in T^*Q$ and $u \in U$, we have
\begin{align*} \tilde H_{\nu, u}(\tilde \lambda) &:= \tilde \lambda \, \tilde {\sf b}(q,u) + \nu \tilde {\sf c}(q,u), \\
& = \tilde \lambda \, h_q {\sf b}(\pi(q),u) + \nu {\sf c}(\pi(q),u) \\
& = \pi^2(\tilde \lambda) \, {\sf b}(\pi(q),u) + \nu {\sf c}(\pi(q),u) = H_{\nu,u}(\pi^{2}(\tilde \lambda)).\end{align*}
The rest follows from Theorem~\ref{th:main} and the fact that $\sigma(\vl \alpha, \newbullet) = - \alpha(p^M_* \newbullet)$ for any $\alpha \in \Gamma(T^*M)$, where $\sigma$ is the canonical symplectic form on $T^*M$.
\end{proof}

\subsection{Sub-Riemannian manifolds and submersions} \label{sec:sR}
Let $M$ be a manifold, let $p^D:D \to M$ be a subbundle of $TM$ and let $\inc:D \to M$ denote the inclusion map. Consider the control system $(\mathsf{a}, \mathsf{b}) = (p^D, \inc)$. Let $\tensorg$ be a metric tensor defined only on $D$, and consider the optimal control problem associated to $(\mathsf{a}, \mathsf{b}, \mathsf{c})$, where
\begin{equation} \label{costSR} \mathsf{c}(v) = \frac{1}{2} \|v\|_{\tensorg}^2 , \qquad v \in D.\end{equation}
Minimal curves can be considered as length-minimizing curves in a sub-Riemannian manifold. A sub-Riemannian manifold is a triple $(M, D, \tensorg)$, where $M$ is a connected manifold, $D$ is a subbundle of $TM$ and $\tensorg$ is a metric tensor on $D$. The distance in this geometry is given by
$$\metricd_{cc}(x_0,x_1) = \inf \left\{ \int_0^1 \|\dot \gamma \|_{\tensorg} \, dt \, \colon \, \gamma \text{ is $D$-horizontal}, \gamma(0) = x_0, \gamma(1) = x_1 \right\}.$$
A sufficient condition for this distance to be finite, i.e. that any pair of points can be connected by a $D$-horizontal curve, is that $D$ is {\it bracket-generating}. This means that vector fields with values in~$D$ along with the iterated brackets span the entire tangent bundle. We will consider this as an optimal control problem.. Since the cost function is as in \eqref{costSR}, we use $L^2$-controls rather than controls in $L^\infty$. For more details on sub-Riemannian manifolds, see \cite{Mon02}.

Normal extremals can be described in the following way. The metric tensor~$\tensorg$ defines a bundle map $\sharp^{\tensorg}:T^*M \to D$ by the identity $\langle \sharp^{\tensorg} \lambda, v\rangle_{\tensorg} = \lambda v \text{ for any } v\in D.$
Normal extremals are then the solution of the Hamiltonian system with Hamiltonian $H_{sR}(\lambda) = \frac{1}{2} \lambda(\sharp^{\tensorg}\lambda)$. Projections  to $M$ of normal extremals are always smooth and are length minimizers locally. We therefore call such curves in $M$ {\it normal geodesics}. 

Abnormal extremals $\lambda(t)$ on a sub-Riemannian manifold also have an alternate description. Let $\sigma$ be the canonical symplectic form on $T^*M$ and let $\lambda(t)$ be a curve in~$T^*M$. Then $\lambda(t)$ is an abnormal extremal if and only if it is an absolutely continuous curve in $\Ann(D)$ with an $L^2$-derivative, never meeting the zero section, such that
$\varsigma(\dot \lambda(t), \newbullet)|_{\Ann(D)}=0$. A curve fulfilling the latter requirement, is often referred to as a {\it characteristic} of $D$. We will use the term {\it abnormal curve} $\gamma(t)$ if the curve is the projection to $M$ of an abnormal extremal $\lambda(t)$. 

Let $\pi:Q \to M$ be a submersion with vertical bundle $\calV$ and a chosen Ehresmann connection $\calH$. We lift the sub-Riemannian structure as with more general optimal control problems. The lifted structure can be described as the sub-Riemannian manifold $(Q, \tilde D, \tilde \tensorg)$, where
$$\tilde D = \left\{ h_q v \, \colon \, v \in D_x, Q \in Q_x, x \in M \right\} = (\pi_*)^{-1}(D) \cap \calH,$$
and $\tilde \tensorg = \pi^* \tensorg |_{\tilde D} .$ We will say that this sub-Riemannian structure on $Q$ is lifted from~$M$. We present the following corollaries of Theorem \ref{th:main}.

\begin{corollary} \label{cor:sRnormal}
A curve $\lambda(t)$ in $T^*M$ is a normal extremal if and only if any sufficently short segment is a projection of a normal extremal in $T^*Q$ contained in $\Ann(\calV)$. Conversely, the image in $T^*M$ of normal extremals in $T^*Q$ satisfy equation \eqref{maineq1} and \eqref{maineq2} with $H = H_{sR}$.
\end{corollary}

In the special case when $D = TM$, making the base space a Riemannian manifold, projections of normal geodesics are given by Corollary~\ref{cor:Riemann}. The top space is then a sub-Riemannian manifold $(Q,\calH, \pi^*\tensorg |_{\calH})$. Also, $\calH$-horizontal lifts of Riemannian geodesics are normal geodesics. The latter fact was first observed in \cite[Theorem 6.2, Corollary 6.5]{KhLe09}.

\begin{proposition} \label{prop:sRabnormal}
A curve $\gamma(t)$ in $M$ is abnormal if and only if any horizontal lift of any sufficiently short segment of the curve is abnormal.

Conversely, a curve $\tilde \lambda(t)$ in $\Ann(\tilde D)$ with an $L^2$-derivative is an abnormal extremal if and only if
$$\xymatrix{ \tilde \lambda(t) \ar@{|->}[d]_{\pr_{\calV}^*} \ar@{|->}[rr]^{\pi^2} \ar@{|->}[rd]^{p^Q} & &  \lambda(t)  \ar@{|->}[d]^{p^M} \\ \beta(t)  & \tilde \gamma(t) \ar@{|->}[r]_{\pi} & \gamma(t) } $$
satisfy
\begin{enumerate}[\rm (i)]
\item $\tilde \gamma(t)$ is $\tilde D$-horizontal (and hence $\gamma(t)$ is $D$-horizontal),
\item $\sigma(\dot \lambda(t), \newbullet) |_{\Ann(D)}  = -\beta(t) \calR(\dot \gamma(t), p^M_* \newbullet),$
\item $\bnabla_{\dot \gamma} \beta(t) = 0,$
\end{enumerate}
and $\lambda(t)$ and $\beta(t)$ do not vanish simultaneously.
\end{proposition}

The proof of this theorem uses elements of the proof of Theorem \ref{th:main}, and is therefore left to Section \ref{app:abnormal}. If we consider the special case when $D = TM$, $\Ann(D)$ contains just the zero-section and the result is written as follows.
\begin{corollary}
Let $\pi:Q \to M$ be a submersion and let $\calH$ be an Ehresmann connection on $\pi$. Consider the sub-Riemannian manifold $(Q, \calH, \tilde \tensorg)$ for some metric tensor $\tilde \tensorg$ on $\calH$. Then a $\calH$-horizontal curve $\tilde \gamma(t)$ in $Q$ is an abnormal curve if and only if there is a non-zero curve $\beta(t) \in \Ann(\calH)_{\tilde \gamma(t)}$ such that 
$$\bnabla_{\dot \gamma} \beta(t) = 0, \qquad \beta(t)\calR(\dot \gamma(t), \newbullet) =0, \qquad \text{ where } \pi(\tilde \gamma(t)) = \gamma(t).$$
\end{corollary}

\section{Optimal control of rolling manifolds} \label{sec:Rolling}
\subsection{Manifolds rolling without twisting or slipping} \label{sec:Rollwithout}
Throughout this section, let $(M, \tensorg)$ and $(\widehat M, \tensorg)$ be two oriented Riemannian manifolds of the same dimension $n \geq 2.$ We want to consider the kinematic system of $M$ rolling on $\widehat M$. Let us first define our configuration space. For two oriented inner product spaces $V$ and $\widehat V$, let $\SO(V, \widehat V)$ be the space of all linear, orientation preserving isomorphisms from $V$ to $\widehat V$ with the obvious induced manifold structure, making it diffeomorphic to $\SO(n)$. Define a fiber bundle $Q$ over $M \times \widehat M$ as
$$Q = \left\{ q \in \SO(T_xM, T_{\widehat x} \widehat M) \, \colon \, x \in M, \widehat x \in \widehat M \right\}.$$
An element $q : T_x M \to T_{\widehat x} \widehat M$ in $Q$ represents a configuration where $M$ lies tangent to $\widehat M$ at the points $x$ and $\widehat x$. Two vectors, one in each tangent space, lie adjacent if one is mapped to the other by $q$.

A rolling is a curve $q(t)$ in this configuration space. We will assume that we have high friction, giving us the constraints that we cannot slip or twist. Let $\pi:Q \to M$ and $\widehat \pi: Q \to \widehat M$ be the respective natural projections to $M$ and~$\widehat M$.

\begin{definition}
An absolutely continuous curve $q(t)$ in $Q$ with $\gamma(t) = \pi(q(t))$ and $\widehat \gamma(t) = \widehat \pi(q(t))$ is a rolling without twisting or slipping if, for almost every $t$,
\begin{enumerate}[$\bullet$]
\item \emph{(No slipping condition)} $q(t) \dot \gamma(t) = \dot {\widehat \gamma}(t),$
\item \emph{(No twisting condition)} Any vector field $X(t)$ along $\gamma(t)$ is parallel if and only if $q(t) X(t)$ is parallel along $\widehat \gamma(t)$.
\end{enumerate}
\end{definition}
Intuitively, the no slipping condition means that if $\gamma(t) = \pi(q(t))$ is a constant curve, then so is $\widehat \gamma(t) = \widehat \pi(q(t))$, while the no twisting condition means that if both $\gamma(t)$ and $\widehat \gamma(t)$ are constant, then so is $q(t)$. These constraints can be described by a subbundle $D$ of $TQ$ of rank $n$. We will describe this subbundle locally and refer to \cite{GGML12} for details.

For any sufficiently small neighborhood $U$ on $M$, choose an orthonormal basis $e_1,\dots, e_n$ on $U$. Let $\widehat U$ be a similar neighborhood in $\widehat M$ with orthonormal basis $\widehat e_1, \dots, \widehat e_n$. We can then trivialize $Q$ over $U \times \widehat U$ by
$$\begin{array}{ccc}
Q|_{U \times \widehat U} & \to & U \times \widehat U \times \SO(n) \\
q \in \SO(T_x M, T_{\widehat x} \widehat M) & \mapsto & \left(x , \widehat x, (q_{rs}) \right)
\end{array} ,$$
where $q_{rs} = \left\langle \widehat e_r|_{\widehat x}, q e_s|_x \right\rangle_{\widehat{\tensorg}}.$ Relative to this trivialization, define vector fields
$$W_{rs} = \sum_{k=1}^n \left(q_{kr} \partial_{q_{ks}} - q_{ks} \partial_{q_{kr}} \right).$$
Then $D$ restricted to $Q|_{U \times \widehat U}$ is spanned by
\begin{equation} \label{BasisQ} e_j + q e_j + \sum_{1 \leq r < s \leq n} \left( \left\langle e_r, \nabla_{e_j} e_s \right\rangle_{\tensorg} - \left\langle qe_r, \nabla_{q e_j} qe_s \right\rangle_{\widehat \tensorg} \right) W_{rs},\end{equation}
for $j = 1, \dots, n$. Here, the $qe_j$ denote the vector fields $q \mapsto qe_j|_{\pi(q)}.$

From now on, we will often omit the phrase ``without twisting or slipping'', only stating that $M$ roll on $\widehat M$, with the constraints being implicit.

\subsection{Rolling along minimal curves as an optimal control problem} \label{sec:SolRolling}
We want to solve the following problem.
\begin{enumerate}[\sf (Opt)] 
\item Given two configurations $q_0$ and $q_1$, find a rolling $q(t)$ from $q_0$ to $q_1$ without twisting or slipping such that $\gamma(t) = \pi(q(t))$ has minimal length.
\end{enumerate}
The no-slip condition ensures that looking for rolling motions which minimizes the length of $\widehat \gamma(t) = \widehat \pi(q(t))$ is an equivalent problem, as the lengths of $\gamma(t)$ and $\widehat \gamma(t)$ must always coincide.

We want to show that the problem {\sf (Opt)} is an optimal control problem lifted from~$M$. Indeed, from the description of $D$ in \eqref{BasisQ}, it follows that $D$ is an Ehresmann connection on $\pi:Q \to M$. Hence, any curve tangent to $D$ is uniquely determined by its starting point and its image in $M$. Furthermore, we are minimizing a cost which only depends on the projection of a rolling $q(t)$ to $M$. In summary, our problem {\sf (Opt)} is a lifting of the problem of finding curves of minimal length in $M$ using the Ehresmann connection $D$. Alternatively, {\sf (Opt)} can be seen as finding minimal curves in the sub-Riemannian manifold $(Q, D, \tensorh)$ where the metric $\tensorh$ on $D$ is defined as
$$\langle v_1, v_2 \rangle_{\tensorh} = \langle \pi_* v_1, \pi_* v_2\rangle_{\tensorg} , \qquad v_1, v_2 \in D.$$
Relative to this metric, the vector fields in \eqref{BasisQ} form a local orthonormal basis. Moreover, $\langle v_1, v_2\rangle_{\tensorh} = \langle \widehat \pi_* v_1, \widehat \pi_* v_2 \rangle_{\widehat \tensorg}$, so the sub-Riemannian structure $(D, \tensorh)$ can be considered as lifted from $\widehat M$ as well.

In general, this problem may not have a solution, since there may be no way to roll from $q_0$ to $q_1$ without twisting or slipping. We therefore assume the following condition.
\begin{enumerate}[\sf (A)]
\item Define $Q_x = \pi^{-1}(x)$. Let $R$ be the curvature tensor of the Levi-Civita connection $\nabla$ on $M$,
$$R(X,Y) = \nabla_X \nabla_Y - \nabla_Y \nabla_X - \nabla_{[X,Y]}, \qquad   X, Y \in \Gamma(TM).$$
Let $\widehat R$ the curvature on $\widehat M$ and define
\begin{equation} \label{qR} \widehat{ R}^q(v_1,v_2) v_3 := q^{-1} \widehat R(q v_1, q v_2) q v_3, \qquad q \in Q_x, v_j \in T_xM. \end{equation}
For any $q \in Q_x$, $x \in M$, we assume that the quadratic form ${\bf s}^q: \bigwedge^2 T_xM \to \real$, given by
$${\bf s}^q(v \wedge w) = \langle R ( v,  w)  v,  w \rangle_{\tensorg}  - \langle \widehat{R}^q( v,  w)  v,  w \rangle_{ \tensorg}$$
is non-degenerate.
\end{enumerate}
From \cite[Theorem 3b, Theorem 5]{Gro12}, we know that if {\sf (A)} holds, then $D$ is bracket-generating and, furthermore, there are no abnormal length minimizers that are not also normal. Hence, we know that a sufficient condition for a rolling $q(t)$ to satisfy {\sf (Opt)} is that it is a projection of an integral curve of the sub-Riemannian Hamiltonian $H_{sR}$ of $(Q, D ,\tensorh).$
We introduce the notation that if $\Lambda \in \bigwedge^2 T_xM$ is a two-vector with decomposition $\Lambda = \sum_{r,s=1}^n \Lambda_{rs} v_r \wedge v_s$ with respect to some basis of $T_xM$, then $R(\Lambda) := \sum_{r,s=1}^n \Lambda_{rs} R(v_r, v_s)$. This definition clearly does not depend on the choice of basis. We use this notation to state our equations for the optimal solutions of the rolling problem.

\begin{theorem} \label{th:optimal}
Let $q(t)$ be a rolling without twisting or slipping with $\gamma(t) = \pi(q(t))$ that solves the problem {\sf (Opt)}. Assume that {\sf (A)} holds. Then there exist a vector field $V(t)$ and a two-vector field $\Lambda(t)$, both along $\gamma(t)$, such that $\gamma(t), V(t)$ and $\Lambda(t)$ are solutions to
$$\nabla_{\dot \gamma} \dot \gamma(t) = R(\Lambda(t)) \dot \gamma(t) - \widehat{R}^{q(t)}( \Lambda(t)) \dot \gamma(t) ,$$
\begin{equation} \label{asymmetry}\nabla_{\dot \gamma} \Lambda(t) = \dot \gamma(t) \wedge V(t), \qquad \nabla_{\dot \gamma} V(t) = \widehat R^{q(t)}( \Lambda(t))  \dot \gamma(t),\end{equation}
where $\widehat R^{q}$ is as in \eqref{qR}.
\end{theorem}
We can write the equations of Theorem~\ref{th:optimal} in $\so(n+1)$. Choose any frame of parallel vector fields $f_1(t), \dots, f_n(t)$ along $\gamma(t)$ and write $\dot \gamma(t) = \sum_{i=1} u_i(t) f_i(t)$, $\Lambda(t) = \sum_{r,s=1}^n \Lambda_{rs}(t) f_r(t) \wedge f_s(t)$ and $V(t) = \sum_{i=1}^n v_i(t) f_i(t).$ Furthermore, write
$$R^{ij}_{rs}(t) = \left\langle R\big(f_i(t),f_j(t)\big) f_r(t), f_s(t) \right\rangle_{\tensorg}, \qquad i,j,r,s =1 ,2, \dots, n.$$
The fact that $R^{ij}_{rs} = R^{rs}_{ij}$ and $R^{ij}_{rs} = - R^{ij}_{sr}$ gives us a linear map $T^R_t: \so(n) \to \so(n)$ by
$$T^R_t:(A_{rs}) \mapsto \left( \sum_{i,j=1}^n R^{ij}_{rs}(t) A_{ij} \right),$$
satisfying $\langle T^R_tA, B \rangle_{\so(n)} = \langle A, T^R_tB \rangle_{\so(n)}$ with respect to $\langle A, B \rangle_{\so(n)} := - \tr AB.$ Define $T^{\widehat R}_t$ analogously with respect to the basis $q(t) f_1(t), \dots, q(t) f_n(t)$. We can consider these maps as endomorphisms on $\so(n+1)$ by
$$T^R_t \left( \begin{array}{cc} A & w \\ -w^\dagger & 0 \end{array} \right) := \left( \begin{array}{cc} T^R_t A & w \\ -w^\dagger & 0 \end{array} \right).$$
Define matrices in $\so(n+1)$ by
$$L(t) = \left( \begin{array}{cc} (\Lambda_{rs}(t) ) & 0 \\ 0 & 0 \end{array} \right), \, \, U(t) = \left( \begin{array}{cc} 0 & u(t) \\ -u(t)^\dagger & 0 \end{array} \right), \, \, Z(t) = \left( \begin{array}{cc} 0 & v(t) \\ -v(t)^\dagger & 0 \end{array} \right),$$
where $u(t) = (u_1(t), \dots, u_n(t))^\dagger$ and $v(t) = (v_1(t) , \dots, v_n(t))^\dagger.$ Then the equations of Theorem~\ref{th:optimal} are given by
\begin{equation} \label{MatrixEquation} \left\{ \begin{array}{c} \dot U(t) = \left[U(t), (T^R_t- T^{\widehat R}_t)L(t)  \right], \\
\dot L(t) = [Z(t), U(t)],  \qquad \dot Z(t) = \left[U(t), T^{\widehat R}_t L(t) \right]. \end{array} \right. \end{equation}
The condition ${\sf (A)}$ is equivalent to $T^{R}_t - T^{\widehat R}_t$ always being an invertible map.
When $M$ and $\widehat M$ are locally symmetric, i.e., when their curvature tensors are parallel, then the maps $T^{R}_t$ and $T^{\widehat R}_t$ are independent of $t$, and \eqref{MatrixEquation} can be solved as matrix equations in $\so(n+1)$.

\begin{remark}
\begin{enumerate}[\rm (a)]
\item
Even though there is an apparent asymmetry between the two manifolds, with the curvature of $\widehat M$, but not of $M$, being present in the equation \eqref{asymmetry}, we could have defined $W(t) = V(t) + \dot \gamma(t)$ and replaced \eqref{asymmetry} with
$$\nabla_{\dot \gamma} \Lambda(t) = \dot \gamma(t) \wedge  W(t), \qquad \nabla_{\dot \gamma} W(t) =  R( \Lambda(t))  \dot \gamma(t).$$
\item
If the condition ${\sf (A)}$ does not hold, it is still true that normal extremals are described in Theorem~\ref{th:optimal}. The only change is that there may be abnormal minimizers and that the problem {\sf (Opt)} may have no solution.  {\sf (A)} is a necessary condition for $D$ to be bracket-generating in the case $n=2$, see \cite{AgSa04,BrHs93}, but not in higher dimensions. For more on controllability of the rolling problem, see~\cite{ChKo12,CGK13,ChKo12b,CGK15,Gro12}. 
\end{enumerate}
\end{remark}


\subsection{Proof of Theorem~\ref{th:optimal}}
By the discussions in Section~\ref{sec:SolRolling} and Corollary~\ref{cor:Riemann}, we know that if $q(t)$ is a solution of {\sf (Opt)} with $\pi(q(t)) = \gamma(t)$, then we have
$$\nabla_{\dot \gamma} \dot \gamma(t) = \sharp \beta (t) \calR(\dot \gamma(t), \newbullet), \qquad \bnabla_{\dot \gamma} \beta(t) = 0.$$
Hence, to finalize the proof, we need to compute the curvature $\calR$ and $\bnabla$-parallel transport. Unfortunately, using local coordinates on $Q$ and the basis \eqref{BasisQ} often lead to complicated calculations. We therefore lift our problem to a larger space $\tilde Q$ using Theorem~\ref{th:main}~(a). This will help us to find the solution.

\subsubsection{Step 1: Lifting of the rolling problem} \label{sec:LiftingProblem}
Let $\nabla$ be the Levi-Civita connection on $(M,\tensorg)$. We will define an oriented orthonormal frame bundle
$$\SO(n) \to \FSO(M) \stackrel{\chi}{\to} M$$
of $M$ as follows. Let $\real^n$ denote the $n$-dimensional euclidean space with its standard basis and orientation. For every $x \in M$, let $\FSO(M)_x$ be the collection of all linear orientation-preserving isometries $f: \real^n \to T_xM$. Each element can equivalently be considered as a choice of a positively oriented orthonormal basis $f_1, f_2, \dots, f_n$ of $T_xM$ by the relation
\begin{equation} \label{MapToBasis} f(r_1,\dots, r_n) = \sum_{i=1}^n r_i f_i, \qquad (r_1, \dots, r_n) \in \real^n.\end{equation}
It is clear that $\FSO(M)_x$ is diffeomorphic to $\SO(n)$ as a manifold, but there is no canonical choice of diffeomorphism. Define a transitive and free right group action of $\SO(n)$ on $\FSO(M)_x$ by $f \cdot a := f \circ a$, where $a \in \SO(n)$ is considered as an endomorphism of $\real^n$. This action allows us to define $\FSO(M) = \coprod_{x \in M} \FSO(M)_x$ as a principal $\SO(n)$-bundle over $M$.

An advantage of the frame bundle is that it has a global basis of vector fields whose brackets are uniquely determined by the curvature $R$. Let $\calH^\nabla$ be the principal connection of $\chi:\FSO(M) \to M$ corresponding to the Levi-Civita connection, that is, $\calH^\nabla$ is the subbundle of $TQ$ defined such that $f(t)$ is tangent to $\calH^{\nabla}$ if and only if $f_1(t), \dots, f_n(t)$ is a parallel frame along $\gamma(t) = \chi(f(t))$. This is well defined since parallel transport with respect to the Levi-Civita connection preserves orientation and orthonormality. The principal connection $\calH^\nabla$ has a canonical choice of basis of vector fields. For any $j =1,2, \dots, n$, define a vector field $X_j$ by
\begin{equation} \label{basisFSO} X_j|_f = h_f f_j, \qquad f \in \FSO(M),\end{equation}
where $h_f$ denotes the horizontal lift to $f$ with respect to $\calH^\nabla$ and the vectors $f_j$ are related to $f$ by~\eqref{MapToBasis}. This gives $T\FSO(M)$ a basis spanned by $X_1, \dots, X_n$ and the vector fields $\xi(A)$, $A \in \so(n)$, where $\xi(A)$ is defined in \eqref{sigmaVF}. 

Let $\omega$ and $\Omega$ be, respectively, the connection form and the curvature form of $\calH^{\nabla}$. The brackets of the vector fields $X_1, \dots, X_n$ and $\xi(A)$, $A \in \so(n)$, are given by
\begin{eqnarray} \label{brackets1} [X_i, X_j] = - \xi(\Omega(X_i, X_j)), & \qquad & [\xi(A), \xi(B)] = \xi([A,B]), \\
 \label{brackets2} [\xi(A), X_j] = \sum_{r=1}^n A_{r j} X_r, & \qquad & A,B \in \so(n), A = (A_{rs}).\end{eqnarray}
The values of $\Omega = (\Omega_{rs})$ in $\so(n)$ are described by
\begin{eqnarray} \nonumber \Omega|_f(hX, hY) =  \Big(- \left\langle R(X,Y) f_r, f_s \right\rangle_{\tensorg} \Big), &\qquad & r, s =1, 2,\dots, n. \end{eqnarray}

The equations \eqref{brackets1} and \eqref{brackets2} are called the Cartan equations when written in terms of forms. Let $\theta = (\theta_1, \dots, \theta_n)^\dagger$ be the $\real^n$-valued one-form on $F(M)$ defined by
$$\theta |_f(\tilde v) = f^{-1} \chi_* \tilde v, \qquad \tilde v \in T_f\FSO(M).$$
We remark that $\theta( \sum_{i=1}^n r_i X_i) = (r_1, \dots, r_n)^\dagger$, while $\theta(\xi(A)) = 0.$ If $\omega = (\omega_{\alpha\beta})$ is the connection form, then both $\theta$ and $\omega$ are constant on vector fields $X_j$, $1\leq j \leq n$, and $\xi(A)$, $A \in \mathfrak{so}$, so by \eqref{brackets1} and \eqref{brackets2} we have
\begin{equation} \label{CartanEquations} d\theta_i = - \sum_{k=1}^n \omega_{ik} \wedge \theta_k, \qquad d \omega_{rs} = - \sum_{k=1}^n \omega_{r k} \wedge \omega_{ks} + \Omega_{rs}.\end{equation}
 For more information on frame bundles, see e.g. \cite[Chapter~2.1]{Hsu02}.

Let $\FSO(\widehat M)$ be the oriented orthonormal frame bundle of $(\widehat M, \widehat \tensorg)$ and write $\tilde Q = \FSO(M) \times \FSO(\widehat M)$. We will use $\chi$ and $\widehat \chi$ for the projections to $M$ and $\widehat M$, respectively. Define a principal $\SO(n)$-bundle over $Q$ by
 $$\begin{array}{rccc}
 \proj : & \tilde Q & \to & Q \\ & (f, \hat f) & \mapsto & \hat f \circ f^{-1}
 \end{array}.
 $$
 
We want to lift our problem to the manifold $\tilde Q$. We will use the following reformulation of rolling without twisting or slipping.
 \begin{lemma}
 \cite[Corollary 1]{Gro12} \label{lemma:lift} Let $q(t)$ be an absolutely continuous curve in $Q$ which projects to the curve $(\gamma(t), \widehat \gamma(t))$ in $M \times \widehat M$. Then $q(t)$ is a rolling without twisting or slipping if and only if it is a projection of an absolutely continous curve $(f(t), \hat f(t))$ in $\tilde Q$ such that
\begin{enumerate}[\rm (a)]
\item $f_1(t), \dots, f_n(t)$ are parallel along $\gamma(t),$
\item $\hat f_1(t), \dots, \hat f_n(t)$ are parallel along $\widehat \gamma(t),$
\item $f(t)^{-1}(\dot \gamma(t)) = \hat f(t)^{-1}(\dot{ \widehat{\gamma}}(t))$ for almost every $t$.
\end{enumerate}
 \end{lemma}
Let $\theta$, $\omega$, $\Omega$, $X_j$ and $\xi(A)$ be defined on $\FSO(M)$ as above and define $\widehat \theta$, $\widehat \omega$, $\widehat \Omega$, $\widehat X_j$ and $\widehat \xi(A)$ similarly on $\FSO(\widehat M)$.  The submersion $\proj: \tilde Q \to Q$ then has vertical bundle $\widetilde \calV := \ker \proj_* = \{\xi(A) + \widehat \xi(A) \, \colon \, A \in \so(n) \}$.
We can then rewrite Lemma~\ref{lemma:lift} to say that an absolutely continuous curve $q(t)$ is a rolling without twisting or slipping if and only if it is the projection of a $\tilde D$-horizontal curve in $\tilde Q$ with
$$\tilde D = \ker \omega \cap \ker \widehat \omega \cap \ker (\theta - \widehat \theta) = \spn \{ X_j + \widehat X_j \, \colon \, j =1, \dots, n\}.$$
In the above formula, the kernels of the three one-forms represent respectively requirement (a), (b) and (c) of Lemma~\ref{lemma:lift} in that order. As $\tilde D$ is a subbundle of the Ehresmann connection
$$\widetilde \calH = \spn \{ X_j, \widehat X_j, \xi(A) - \widehat \xi(A) \, \colon \,  \, A \in \so(n), \, j=1, \dots, n \}$$
on $\proj$, it follows that $\tilde D$ is the horizontal lift of $D$ with respect to this mentioned connection.

From Theorem~\ref{th:main}~(a), we know that rather than solving the Hamiltonian systems of $H_{sR}$ on $Q$, we can look for integral curves of $\widetilde H = H_{sR} \circ \proj^2 = H_{\tensorg} \circ \chi^2$ which are contained in $\Ann(\tilde \calV)$. We remark that $\proj^2$ is defined relative to the Ehresmann connection $\widetilde \calH$, while $\chi^2$ is defined relative to $\tilde D$ on $\chi: \tilde Q \to M$. Using Corollary~\ref{cor:Riemann}, we know that our solutions are of the form
$$\nabla_{\dot \gamma} \dot \gamma(t) = \sharp \beta(t) \calR(\dot \gamma(t), \newbullet), \qquad \bnabla_{\dot \gamma} \beta(t) = 0,$$
where $\calR$ and $\bnabla$ are now defined with respect to $\tilde D$ and $\beta(t) \in \Ann(D) \cap \Ann(\tilde \calV)$.

\subsubsection{Step 2: Curvature and parallel transport of $\tilde D$} 
Since $\beta(t)$ is a curve in $\Ann(\widetilde D) \cap \Ann(\tilde \calV)$, we can write
$$\beta(t) = \sum_{j=1}^n v_j(t) (\theta_j - \widehat \theta_j)|_{f(t), \hat f(t)} + \sum_{r, s =1}^n \Lambda_{rs}(t) (\omega_{rs} - \widehat \omega_{rs})|_{f(t), \hat f(t)},$$
where $\Lambda_{rs}(t) = - \Lambda_{sr}(t)$. We observe that for any $Y \in \Gamma(TM)$,  we have $hY = \sum_{j=1}^n \langle f_i, Y \rangle_{\tensorg} (X_i + \widehat X_i)$ and so by \eqref{brackets1} and \eqref{brackets2},
\begin{eqnarray*}
\bnabla_Y \widehat X_j |_{f,\hat f} & = & - \sum_{i=1}^n \langle f_i, Y \rangle_{\tensorg} \,  \widehat \xi(\widehat \Omega(X_i, X_j)) |_{f,\hat f} \, , \\
\bnabla_Y \xi(A)  |_{f,\hat f} = - \bnabla_Y \widehat \xi(A) |_{f,\hat f} & = & - \sum_{i,j=1}^n \langle f_i, Y\rangle_{\tensorg} \,  A_{ij} \widehat X_j |_{f,\hat f}\, ,
\end{eqnarray*}
and hence
\begin{eqnarray*}
\bnabla_Y (\theta_j - \widehat \theta_j) |_{f,\hat f} & = & - \sum_{i=1}^n \langle f_i, Y \rangle_{\tensorg} (\omega_{ij} - \widehat \omega_{ij})  |_{f,\hat f} \, , \\
\bnabla_Y (\omega- \widehat \omega) |_{f,\hat f}  & = & -\sum_{i,j=1}^n \langle f_i, Y\rangle_{\tensorg} \, \widehat \Omega(\widehat X_i, \widehat X_j) \widehat \theta_j|_{f,\hat f} \, .
\end{eqnarray*}
This allows us to write
\begin{align*}
0 = \bnabla_{\dot \gamma} \beta(t) = &\sum_{j=1}^n \left( \dot v_j(t) + \sum_{i,r,s =1}^n \langle f_i(t), \dot \gamma(t) \rangle_{\tensorg} \, \Lambda_{rs}(t) \widehat \Omega_{rs}(\widehat X_i, \widehat X_j)  \right)  (\theta_j- \widehat \theta_j)|_{f(t), \hat f(t)} \\
& + \sum_{r,s=1}^n \left( \dot \Lambda_{rs}(t) - \langle f_r(t), \dot \gamma(t) \rangle_{\tensorg} \, v_{s}(t) \right) (\omega_{rs} - \widehat \omega_{rs})|_{f(t), \hat f(t)}.
\end{align*}

Define a vector field $V(t)$ and a two-vector field $\Lambda(t)$ along $\gamma(t)$ by respectively $V(t) = \sum_{j=1}^n v_i(t) f_j(t)$ and $\Lambda(t) = \sum_{r,s=1}^n \Lambda_{rs}(t)  f_r(t) \wedge f_s(t).$ We see that
\begin{align*}
\nabla_{\dot \gamma} \Lambda(t) &= \sum_{r,s=1}^n \langle f_r(t), \dot \gamma(t) \rangle_{\tensorg} \, v_s(t) f_r(t) \wedge f_s(t) = \dot \gamma(t) \wedge V(t), \\
\nabla_{\dot \gamma}  V(t) & = - \sum_{i,j,r,s=1}^n \langle f_i(t), \dot \gamma(t) \rangle_{\tensorg} \, \Lambda_{rs}(t) \widehat \Omega_{rs}(\widehat X_i, \widehat X_j) f_j(t)  \\
& = \sum_{i,j,r,s=1}^n \langle f_i(t), \dot \gamma(t) \rangle_{\tensorg} \, \Lambda_{rs}(t) \left\langle \widehat{R}(\hat f_i , \hat f_j) \hat f_r \wedge \hat f_s \right\rangle_{\widehat \tensorg} f_j(t) \\
& = q(t)^{-1} \widehat{R}(q(t) \Lambda(t)) q(t) \dot \gamma(t) .
\end{align*}
Finally, since the curvature $\calR$ of $\widetilde D$ is determined by
$$\calR(X_i + \widehat X_i, X_j + \widehat X_j) = - \xi(\Omega(X_i, X_j)) - \widehat \xi(\widehat \Omega(\widehat X_i, \widehat X_j)),$$
we obtain
\begin{align*}
& \nabla_{\dot \gamma} \dot \gamma(t) = \sharp \beta(t) \calR(\dot \gamma(t),\newbullet ) \\
= & - \sum_{i,j,r,s=1}^n \langle f_i(t) , \dot \gamma(t) \rangle_{\tensorg} \Lambda_{rs}(t) \left( \Omega_{rs}(X_i, X_j) - \widehat \Omega_{rs}(\widehat X_i, \widehat X_j) \right) f_j(t) \\
= & R(\Lambda(t)) \dot \gamma(t) - q(t)^{-1} \widehat{R}(q(t) \Lambda(t)) q(t) \dot \gamma(t).
\end{align*}
This completes the proof.

\subsection{Lower dimensional cases} We look at the special case for solutions of the rolling problem in two and three dimensions. 

\subsubsection{The two-dimensional case} 
Since $\|\dot \gamma\|_{\tensorg}^2$ is constant along solutions of {\sf (Opt)}, we can without loss of generality assume that $e_1(t) := \dot \gamma(t)$ is a unit vector field.
Since $M$ is an oriented manifold, there exists a unique unit vector field $e_2(t)$ such that $\{e_1(t), e_2(t)\}$ is a positively oriented orthonormal basis along $\gamma(t)$. Let $K(x)$ and $\widehat K(\widehat x)$ be the Gaussian curvatures of respectively~$M$ and~$\widehat M$. Our assumption {\sf (A)} is that $K - \widehat K$ never vanishes.

Let $\Lambda(t)$ and $V(t)$ be as in Theorem~\ref{th:main}. Define functions $\Lambda(t), v_1(t)$ and $v_2(t)$ by
$$\Lambda(t) = \Lambda(t) e_1(t) \wedge e_2(t), \qquad V(t) = v_1(t) e_1(t) + v_2(t) e_2(t).$$
Let $\kappa(t)$ be the geodesic curvature of $\gamma(t)$, defined by
$$\kappa(t) = \langle \nabla_{\dot \gamma} e_1(t), e_2(t)\rangle_{\tensorg} = - \langle e_1(t), \nabla_{\dot \gamma} e_2(t) \rangle_{\tensorg}.$$
\begin{corollary} \label{cor:2dim}
\begin{enumerate}[\rm (a)]
\item If $\widehat K$ is constant, then $v_1(t)^2 + v_2(t)^2 + \widehat K \Lambda(t)^2$ is also constant.
\item If $\widehat K \equiv 0$, then there are constants $A$ and $\theta_0$ such that $\theta(t) = \theta_0 + \int_0^t \kappa(s) ds$ is a solution to
$$\frac{d}{dt} \left( \frac{\dot \theta(t)}{K(t)} \right) = A \sin \theta(t), \qquad K(t) := K(\gamma(t)).$$
\item If both $K$ and $\widehat K$ are constant, then there are constants $A$ and $\theta_0$ such that $\theta(t) = \theta_0 + \int_0^t \kappa(s) ds$ is a solution to
$$\ddot \theta(t) = A \sin \theta(t).$$ 
\end{enumerate}
\end{corollary}
We note that Corollary~\ref{cor:2dim}~(c) appeared in \cite{Jur93} for the special case of a ball rolling on a plane.

\begin{proof}
Define, for now, $\theta(t) = \int_0^t \kappa(s) ds$. Write $K(t) = K(\gamma(t))$ and $\widehat K(t) = \widehat K(\widehat \gamma(t))$. By Theorem~\ref{th:optimal},
\begin{align*} \dot \theta(t) = \langle \nabla_{\dot \gamma} \dot \gamma(t) , e_2(t)\rangle_{\tensorg} &= - \Lambda(t) (K(t) - \widehat K(t)), \qquad \dot \Lambda(t) = v_2(t), \\
\langle \nabla_{\dot \gamma} V(t) , e_1(t)\rangle_{\tensorg} &= \dot v_1(t) - \dot \theta(t) v_2(t) = 0 , \\
\langle \nabla_{\dot \gamma} V(t), e_2(t) \rangle_{\tensorg} & = \dot v_2(t)  + \dot \theta(t) v_1(t) = - \Lambda(t) \widehat K(t) .\end{align*}
We remark that from these equations,
$$\frac{d}{dt} \left(v_1(t)^2 + v_2(t)^2 - \widehat K(t) \Lambda(t) \right) = \dot {\widehat K}(t) \Lambda(t)^2,$$
which gives us (a). Since we know $K(t) - \widehat K(t) \neq 0$, we can write
$$\frac{d}{dt} \left(\frac{\dot \theta(t)}{K(t) - \widehat K(t)} \right) = - v_2(t),$$
$$\dot v(t) = \left( \begin{array}{cc} 0 & - \dot \theta(t) \\ \dot \theta(t) & 0 \end{array} \right) v(t) + \frac{\widehat K(t)}{K(t) - \widehat K(t)} \left( \begin{array}{c} 0 \\  \dot \theta(t) \end{array} \right),$$
where $v(t) = (v_1(t), v_2(t))^\dagger$. We look at the special cases of (b) and (c).

If $\widehat K(t) = 0$, then $v_2(t)$ can be written on the form $A \sin(\theta(t) + \theta_0)$ and the result follows. If the curvatures are both constant, we have
$$\dot v(t) = \left( \begin{array}{cc} 0 & - \dot \theta(t) \\ \dot \theta(t) & 0 \end{array} \right) \left( \begin{array}{c} v_1(t) + \frac{\widehat K}{K - \widehat K} \\ v_2(t) \end{array} \right) .$$
It again follows that any solution of $v_2(t)$ can be written as $A/ (K - \widehat K) \sin (\theta(t) + \theta_0).$
Finally, redefine $\theta(t)$ as $\theta(t) = \theta_0 + \int_0^t \kappa(s) ds$.
\end{proof}

\subsubsection{The three-dimensional case}
On a Riemannian three-dimensional manifold $M$, we can identify two-vectors with vectors and hence also define a cross product on $M$ using the Hodge star map $\star$. Define $v \times w = \sharp \star (\flat v \wedge \flat w)$ and note the relation $\langle v \times w_1, w_2 \rangle_{\tensorg} = - \langle w_1, v \times w_2 \rangle_{\tensorg}$. The curvature $R$ gives us a symmetric bilinear tensor $\tensorr$ on the tangent bundle, uniquely determined by the relation
$$\langle R(v_1, w_1) v_2, w_2 \rangle_{\tensorg} = -\langle v_1 \times w_1, v_2 \times w_2 \rangle_{\tensorr}.$$
The sign is chosen such that $\tensorr$ is positive definite if and only if $(M, \tensorg)$ has positive sectional curvature. By slight abuse of notation, we will also use $\tensorr:TM \to TM$ for the endomorphism of the tangent bundle, determined by the relation
$$\langle v,w \rangle_{\tensorr} =\langle v,  \tensorr(w) \rangle_{\tensorg}.$$
Define $\widehat \tensorr$ similarly on $\widehat M$. Assumption {\sf (A)} will in this case mean that
$$v,w \mapsto \langle v, \tensorr(w) \rangle_{\tensorg} - \langle v, q^{-1} \widehat \tensorr(q w) \rangle_{\tensorg}, \qquad v,w \in T_xM,$$
is a non-degenerate bi-linear map for any $x \in M$ and linear isometry $q \in Q_x.$

We now look at the equations in Theorem~\ref{th:optimal}. Write $X(t) = \sharp \star \flat \Lambda(t)$ for the identification of the two-vector field $\Lambda(t)$ with a vector field~$X(t)$ along the curve~$\gamma(t)$. Then we have equations.
$$\nabla_{\dot \gamma} \dot \gamma(t) = \dot \gamma(t) \times \left(\tensorr (X(t)) - q(t)^{-1} \widehat \tensorr (q(t) X(t)) \right),$$
$$\nabla_{\dot \gamma} X(t) = \dot \gamma(t) \times V(t), \qquad \nabla_{\dot \gamma} V(t) = \dot \gamma(t) \times q(t)^{-1} \widehat \tensorr( q(t)X(t)).$$

\subsection{Rolling on a space form} \label{sec:SpaceForm}
Let $\widehat M= \Sigma_{n,\widehat K}$ denote the simply connected, complete $n$-dimensional Riemannian manifold of constant sectional curvature. In other words, $\Sigma_{n,0}$ is isometric to $\real^n$ with the euclidean metric, $\Sigma_{n,\widehat K}$ is an $n$-dimensional sphere of radius ${\widehat K}^{-1/2}$ when $\widehat K > 0$ and if $\widehat K < 0$ then $\Sigma_{n,\widehat K}$ is isometric to an appropriately scaled version of the hyperbolic space. When rolling on such spaces, the configuration space $Q$ can be given a principal bundle structure by \cite[Proposition 4.1]{ChKo12b}. We will describe briefly how our equations look with respect to this structure.

If $G = \Isom(\Sigma_{n, \widehat K})$ is the isometry group of $ \Sigma_{n,\widehat K}$, then there exists a left action of $G$ on $\FSO(\Sigma_{n,\widehat K})$ defined by
$$f\in \FSO(\Sigma_{n, \widehat K})_{\widehat x} \mapsto \varphi \cdot f : = \varphi_* \circ f \in \FSO(\Sigma_{n, \widehat K})_{\varphi(\widehat x)}.$$
In the case of $\Sigma_{n, \widehat K}$, this is a transitive action. Choose a point $o \in \Sigma_{n,\widehat K}$ and a reference frame $\hat f^o \in \FSO(\Sigma_{n,\widehat K})_o$. Then we may identify $\FSO(\Sigma_{n,\widehat K})$ with $G$ by identifying $\hat f = \varphi \cdot  \hat f^o$ with $\varphi$. Furthermore, we may identify the right action of $\SO(n)$ on $\FSO(\Sigma_{n,\widehat K})$ by the right action on $G$ by the stabilizer $G_o$ of $o \in \Sigma_{n, \widehat K}$. Indeed, if $\hat f = \varphi \cdot \hat f^o$, then $\hat f \cdot a = \varphi \cdot \varphi_a \cdot \hat f^o$ where $\varphi_a$ is the unique isometry satisfying
\begin{equation} \label{SOGo} \varphi_a(o) = o, \qquad \varphi_a \cdot \hat f^o = \hat f^o \cdot a.\end{equation}

Let $\frakg$ be the Lie algebra of $G$ and let $\frakk$ be the sub-algebra corresponding to the subgroup $G_o$. Since isometries preserve parallel transport with respect to the Levi-Civita connection, the vector fields $\widehat X_j$ are preserved under the action of $G$, i.e., they are left invariant vector fields. Hence, if we write $\frakp = \spn_{1 \leq j \leq n} \{ \widehat X_1|_{\hat f^o}, \dots, \widehat X_1|_{\hat f^o} \}$, then $\frakg = \frakp \oplus \frakk$ and this direct sum corresponds to the splitting $\widehat \calH^\nabla \oplus \ker \widehat \chi$. We refer to \cite[Chapter 9]{ONe83} for the details, noting that we may identify $\frakg$ with the Lie algebra matrices on the form
$$\left( \begin{array}{cc} A & v \\ - \widehat K v^\dagger & 0 \end{array} \right), \qquad A \in \so(n), v \in \real^n.$$
In this case, $\frakp$ and $\frakk$ can be identified with the subspaces $A =0$ and $v =0$ respectively. Note that the cases $\widehat K$ equal to $1$,$ -1$ or $0$ correspond to the cases where $\frakg$ may be identified with $\so(n+1)$, $\so(1,n+1)$ and $\se(n)$, respectively.

Consider now the case of an $n$-dimensional Riemannian manifold~$M$ rolling on~$\Sigma_{n,\widehat K}$ without slipping or twisting. Then we may identify $\tilde Q = \FSO(M) \times \FSO(\Sigma_{n,\widehat K})$ with $\FSO(M) \times G$ and we can consider $\chi: \FSO(M) \times G \to M$ as a $\widetilde G$-principal bundle with the action of $\widetilde G = \SO(n) \times G$ on the right. The action of this principal bundle is given by
$$(f, \varphi) \cdot (a, \varphi_1) = (f \cdot a, \varphi_1^{-1} \cdot \varphi \cdot \varphi_a), \qquad (a, \varphi_1) \in \SO(n) \times G = \widetilde G.$$
The subbundle $\widetilde D$ is invariant under the action of $\widetilde G$. By dividing out by elements $(a,\id_{\Sigma_{n,\widehat K}}) , a \in \SO(n)$, we get that $D$ is a principal connection on the principal $G$-bundle $\pi: Q \to M$. However, it is still simpler to do our computations on $\FSO(M) \times G$.

Define $\xi$ and $\widehat \xi$ as in Section~\ref{sec:LiftingProblem}, and define $\xi^{\widetilde G}$ as
$$\xi^{\widetilde G}(A,B)|_{f,\varphi} = \left. \frac{d}{dt} (f,\varphi) \cdot (\exp(tA), \exp(tB)) \right|_{t=0}$$
for any $A \in \so(n), B \in \frakg$ and $(f,\varphi) \in \FSO(M) \times G$. Note that
$$\xi^G(A,B) |_{f,\varphi} = \xi(A) |_{(f,\varphi)} + \widehat \xi(A) |_{(f,\varphi)} -  B \cdot \varphi.$$
Furthermore,
$$\widehat \xi(A)|_{f,\varphi} = \varphi \cdot \left( \begin{array}{cc} A & 0 \\ 0 & 0 \end{array} \right).$$
 Write $\omega^{\widetilde G}$ and $\Omega^{\widetilde G}$ for the connection form and the curvature form of $\widetilde D$, respectively. Then the values of $\omega^{\widetilde G}$ in $\so(n) \oplus \frakg$ are determined by the relations
$$\omega^{\widetilde G}\left(\sum_{i=1}^n v_i \widehat X_i |_{f,\varphi} \right) =- \omega^{\widetilde G}\left( \sum_{i=1}^n v_i X_i |_{f,\varphi} \right) = 0 \times \left(\begin{array}{cc} 0 & v \\ - \widehat K v^\dagger & 0 \end{array} \right),$$
$$\omega^{\widetilde G}(\xi(A)|_{f,\varphi}) = A \times  \Ad(\varphi) \left(\begin{array}{cc} A & 0 \\ 0 & 0 \end{array} \right),$$
$$\omega^{\widetilde G}(\widehat{\xi}(A)|_{f,\varphi}) = 0 \times - \Ad(\varphi)\left(\begin{array}{cc} A & 0 \\  0 & 0 \end{array} \right).$$
As a consequence, $\Omega^{\widetilde G}$ is given by
$$\Omega^{\widetilde G}(X_i + \widehat X_i, X_j + \widehat X_j) = \Omega(X_i, X_j) \times  \Ad(\varphi)\left( \begin{array}{cc} \Omega(X_i, X_j) - \widehat \Omega(\widehat X_i, \widehat X_j) & 0 \\ 0 & 0 \end{array}  \right).$$

We next use the formulas of Example~\ref{ex:Gauge}. Let $q(t)$ be an optimal solution to the rolling problem with $q(0) = q_0$ with projection $\gamma(t)$ in $M$. Let $(f^o,\hat f^o) \in \FSO(M) \times \FSO(\Sigma_{n,\widehat K})$ be such that $q_0 = \hat f^o \circ (f^o)^{-1}$. Use $\widehat f^o$ as the initial frame to identify $\FSO(\Sigma_{n,\widehat K})$ with $G = \Isom(\Sigma_{n,\widehat K})$. Let $\beta(t)$ be contained in $\Ann(\widetilde \calV)$ and satisfying $\bnabla_{\dot \gamma} \beta(t) =0$, where $\widetilde \calV = \ker \proj_*$. Then $\beta(t)$ can be written, with slight abuse of notation, as $\beta \pr_{\frakg} \omega^{\widetilde G}(\newbullet)|_{\gamma(t)}$ where $\pr_{\frakg}$ is the projection $\mathfrak{\so}(n) \times \frakg \to \frakg$ and $\beta \in \frakg^*$ is constant.

Let $f(t)$ be defined by parallel transport of $f^o$ along $\gamma(t)$ and let $\varphi(t)$ be defined such that $q(t) f(t) = \varphi(t) \cdot \hat f^o$. If $u(t) = f(t)^{-1}(\dot \gamma(t))$ then
\begin{equation} \label{PhiDer} \dot \varphi(t) = \varphi(t) \left( \begin{array}{cc} 0 & u(t) \\ - \widehat K u(t)^\dagger & 0 \end{array} \right) .\end{equation}
Furthermore, by Example~\ref{ex:Gauge}, we know that $\gamma(t)$ is given as the solution of
\begin{align*} \label{principalSol}
\nabla_{\dot \gamma} \dot \gamma(t) = & - \sharp \beta \pr_{\frakg} \Omega^{\widetilde G}(h_{f(t),\varphi(t)} \dot \gamma(t), \newbullet) \\  \nonumber
= & - \sum_{i,j=1}^n \langle \dot \gamma(t), f_i(t)\rangle_{\tensorg} \beta \Ad(\varphi(t)) (\Omega(X_i, X_j) - \widehat \Omega(\widehat X_i, \widehat X_j)) f_j(t).
\end{align*}
Hence, if $u(t) = f(t)^{-1}(\dot \gamma(t))$ and $\Omega^{ij}(t) = \Omega(X_i, X_j)|_{f(t)}$, we can write
\begin{equation} \label{GaugeForm} \dot u_j(t) = \sum_{i=1}^n u_i(t) \beta \Ad(\varphi(t)) \left( \begin{array}{cc} - \Omega^{ij}(t) + \widehat K W_{ij} & 0 \\ 0 & 0 \end{array}\right), \qquad \beta \in \frakg^*,\end{equation}
where $W_{ij}$ is the matrix whose coefficient in row $r$ and column $s$ is given by $\delta_{r,i} \delta_{s,j} - \delta_{s,i} \delta_{r,j}$. We leave it to the reader to verify that the equations \eqref{PhiDer} and \eqref{GaugeForm} can be transformed into the equations \eqref{MatrixEquation}.

\subsection{Comparison to elastica}
Consider the following problem. Let $M$ be an oriented Riemannian manifold with oriented orthonormal frame bundle $\FSO(M)$. We want to consider the following problem. Let $e(t)$ be a curve in $\FSO(M)$, given by the vector fields $e_1(t), \dots, e_n(t)$ along a curve $\gamma(t)$.  Assume that $\gamma(t)$ is parametrized by arc length and that $\dot \gamma(t)= e_1(t)$. Given initial and final conditions, we want to find the curves $e:[0,T] \to \FSO(M)$ such that the following cost functional is minimized:
\begin{equation} \label{costE} E\big[e(\newbullet)\big] = \frac{1}{2} \sum_{i=1}^n \int_0^T \| \nabla_{\dot \gamma} e_i(t) \|^2_{\tensorg} \, dt.\end{equation}
Using the approach of \cite{JuMP02} we solve this problem, see also \cite[Section~7]{Jur11}.
\begin{proposition} \label{prop:MinimalFrame}
If $\gamma(t)$ is the projection of a minimizer in the above problem, then there exist a two-vector field $\Lambda(t)$ and  vector field $W(t)$ along $\gamma(t)$ such that
$$\nabla_{\dot \gamma} \dot \gamma(t) =  \sharp \iota_{\dot \gamma(t)} \flat \Lambda(t),$$
$$\nabla_{\dot \gamma} \Lambda(t) =  \dot \gamma(t) \wedge W(t), \qquad \nabla_{\dot \gamma} W(t) = R(\Lambda(t)) \dot \gamma(t).$$
Furthermore,  $\nabla_{\dot \gamma} e_k(t) = \sharp \iota_{e_k(t)} \flat \Lambda(t).$
\end{proposition}
Let us first remark that solutions of the above problem coincide with solutions of {\sf (Opt)} for the special case when $M$ and $\widehat M$ both have constant sectional curvature, i.e., if $R(\Lambda(t))v = - K \sharp \iota_v \flat \Lambda(t)$ and $\widehat R(\Lambda(t))v = - \widehat K \sharp \iota_v \flat \Lambda(t)$ for some number $K$ and~$\widehat K$. Hence, in this case, finding a minimal rolling coincides with finding a ``most parallel'' orthonormal frame. In general, solutions of the above problem and {\sf (Opt)} do not coincide, even in the locally symmetric case. It would be interesting to investigate what curves (in addition to geodesics) that both appear as solutions of the ``most parallel'' orthonormal frame problem and the optimal rolling problem for general manifolds $M$ and $\widehat M$, but this goes beyond the scope of this paper.

The problem of a frame $e(t)$ minimizing cost $E$ in \eqref{costE} is linked to the problem of finding \emph{elastica} in $M$. From the equations $\nabla_{\gamma} e_k(t) = \sharp \iota_{e_k} \flat \Lambda(t)$ and $\nabla_{\dot \gamma} \Lambda(t) = \dot \gamma(t) \wedge W(t)$, it follows that $\langle e_r, \nabla_{\dot \gamma} e_s \rangle_{\tensorg}$ is constant whenever $r ,s \geq 2$. If all of these constants vanish, then the curve $\gamma(t)$ is an elastic in $M$. We again refer to \cite{JuMP02} and \cite[Section 7]{Jur11} for details. See also \cite{Zim05,JuZi08}.

\begin{proof}[Proof of Proposition~\ref{prop:MinimalFrame}]
Let the vector fields $X_1, \dots, X_n$ be as in \eqref{basisFSO}. Note that for any curve $e(t)$ in $\FSO(M)$ with projection~$\gamma(t)$, we have
\begin{eqnarray*} \dot e(t) & =&  \sum_{i=1}^n \langle \dot \gamma(t), e_i(t) \rangle_{\tensorg} X_i|_{\gamma(t)} + \xi(A(t)), \\ A(t) &=& (A_{rs}(t)) = \Big( \langle e_r(t), \nabla_{\dot \gamma} e_s(t) \rangle_{\tensorg}\Big).\end{eqnarray*}
We can consider this as an optimal control problem
$$\begin{array}{m{3cm}m{4cm}}  \tiny \xymatrix{ \FSO(M) \times \so(n) \ar[rr]^{\sf b} \ar[rd]_{\sf a} & & T\FSO(M) \ar[ld]^{p^{T\FSO(M)}} \\ & \FSO(M) } &${\sf b}(f, A) = X_1|_f + \xi(A)|_f,$ \end{array}$$
where we minimize the cost functional of the cost function
$${\sf c}(f,A) = \frac{1}{2} \langle A, A \rangle_{\so(n)}, \qquad \langle A, B \rangle_{\so(n)} = - \tr AB =  \sum_{r,s=1}^n A_{rs} B_{rs}.$$
We have the corresponding PMP-Hamiltonian,
$$H_{\nu,A}(\lambda) = \lambda(X_1) + \lambda(\xi(A)) + \frac{\nu}{2} \langle A, A \rangle_{\so(n)}.$$
If we give $T^*\FSO(M)$ the coordinates
$$\lambda = \sum_{i=1}^n \lambda_i \theta_i |_x + \sum_{r,s=1}^n \Lambda_{rs} \omega_{rs} |_x, \quad \Lambda_{rs} = - \Lambda_{sr}, \quad  \lambda \in T^*M,$$
then by \eqref{CartanEquations}, we get
\begin{align*} \vec{H}_A |_\lambda & = X_1|_\lambda + \xi(A)|_\lambda + \sum_{j=1}^n \left( - \sum_{r,s=1}^n \Lambda_{rs} \Omega_{rs}^{1j}(x) + \sum_{k=1}^n \lambda_k A_{kj} \right) \partial_{\lambda_j} \\
& \quad + \sum_{r=1}^n \frac{\lambda_r}{2} \left( \partial_{\Lambda_{1r}} - \partial_{\Lambda_{r1}} \right) - 2\sum_{k,r,s=1}^n A_{rk} \Lambda_{ks} \partial_{\Lambda_{rs}} ,\end{align*}
where $\Omega^{ij} = (\Omega_{rs}^{ij}) = \Omega(X_i, X_j) $.

Any abnormal extremal $\lambda(t)$ with projection $\gamma(t)$ must by Theorem~\ref{th:PMP}~(iii) be on the form $\lambda(t) = \sum_{j=2}^n \lambda_j(t) \theta_j |_{\gamma(t)}$. Furthermore, Theorem~\ref{th:PMP}~(ii) tells us that
$0 = \dot \Lambda_{1r}(t) = \lambda_r(t)/2$ for any $1 \leq r \leq n$, so $\lambda_r(t) = 0$. Hence, there are no abnormal extremals.

Turning to normal extremals and choosing $\nu =-1$, if $(\gamma(t), A(t))$ is the optimal control corresponding to normal extremal $\lambda(t)$, then $A(t) =  (\Lambda_{rs}(t))$ by (iii) in Theorem~\ref{th:PMP}~(iii). Finally, it follows from Theorem~\ref{th:PMP}~(ii) that
$$\dot \gamma(t) = e_1(t), \qquad \langle e_r(t), \nabla_{\dot \gamma} e_s(t) \rangle_{\tensorg} = \Lambda_{rs}(t),$$
$$\dot \lambda_j(t) = \sum_{k=1}^n \lambda_k(t) \Lambda_{kj}(t) - \sum_{r,s=1}^n \Lambda_{rs}(t) \Omega^{1j}_{rs}(\gamma(t)) ,$$
$$\dot \Lambda_{rs}(t) = - 2\sum_{k=1}^n \Lambda_{rk}(t) \Lambda_{ks}(t) - \frac{1}{2} \Big(\lambda_r(t) \delta_{1,s} - \lambda_s(t) \delta_{1,r} \Big).$$
Define $\Lambda(t) = \frac{1}{2} \sum_{r,s=1}^n \Lambda_{rs}(t) e_r(t) \wedge e_s(t)$ and $W(t) = \frac{1}{2} \sum_{j=1}^n \lambda_k(t) e_k$ for the result. Note that
$$\Lambda(t) =  \frac{1}{2} \sum_{r,s=1}^n \Lambda_{rs}(t) e_r(t) \wedge e_s(t) = \frac{1}{2}\sum_{k=1}^n e_k(t) \wedge \nabla_{\dot \gamma} e_k(t).$$
It follows that $\nabla_{\dot \gamma} e_k(t) = \sharp \iota_{e_k(t)} \flat \Lambda(t).$
\end{proof}

\section{Proofs} \label{sec:proofs}
\subsection{Proof of Theorem \ref{th:main} and necessary theory} \label{sec:proofmain}
The idea of the proof is as follows. Let $\pi^2:T^*Q \to T^*M$ be the submersion defined in Section~\ref{sec:SubEhresmann} relative to some Ehresmann connection $\calH$ on $\pi:Q \to M$. Elements in $T^*Q$ and $T^*M$ are denoted by $\tilde \lambda$ and $\lambda$, respectively. Let $\theta$ be the Liouville one-form on $T^*M$, i.e.,
$$\theta(w) = \lambda(p^M_* w), \qquad w \in T_\lambda T^*M,$$
and use $\sigma = - d\theta$ for the canonical symplectic form. Let $\tilde \theta$ and $\tilde \sigma$ be defined similarly on $T^*Q$. If we look at the Hamiltonian vector fields of $H$ and $\tilde H = H \circ \pi^2$, they are related by
\begin{equation} \label{prelift} \tilde \sigma\left( \vec{\tilde H}, \tilde u \right) = \sigma( \vec{H}, \pi^2_* \tilde u), \qquad \tilde u \in T(T^*Q).\end{equation}
Corresponding to the relation \eqref{prelift}, there is actually a ``symplectic lifting map''
$$\liftmap_{\tilde \lambda} : T_{\lambda} T^*M \to T_{\tilde \lambda} T^*Q, \qquad \pi^2(\tilde \lambda) = \lambda,$$
on each tangent space. In Section~\ref{sec:lifted} we will show that the maps $\liftmap_{\tilde \lambda}$ give us an Ehresmann connection $\calH^2$ on $\pi^2:T^*Q \to T^*M$. We will show that the $S_{\tilde \lambda}$ are not the horizontal lifts of this Ehresmann connection in Section~\ref{sec:h2} , since $\pi^2_* \, \liftmap_{\tilde \lambda}$ is not the identity in general. However, the difference between $\liftmap_{\widetilde \lambda}$ and the true horizontal lift is exactly described by the curvature $\calR$ of $\calH$. We can use this fact to complete the proof in Section~\ref{sec:proofmainsec}. A description in local coordinates is included in Section~\ref{sec:localcoordinates}.

\subsubsection{An induced Ehresmann connection on $\pi^2$} \label{sec:lifted}
For any $\tilde \lambda \in T^*Q$ with $\pi^2(\tilde \lambda) = \lambda$, let $\liftmap_{\tilde \lambda}: T_\lambda T^*M \to T_{\tilde \lambda} T^*Q$, $w \mapsto S_{\tilde \lambda} w$ be the map defined such that 
$$\tilde \sigma(\liftmap_{\tilde \lambda} w, \tilde u) = \sigma(w, \pi^2_* \tilde u) \qquad \text{for any } \tilde u \in T_{\tilde \lambda} T^* Q.$$
In other words, let $I : T T^*M \to T^*T^*M$ be the symplectic isomorphism $I(w) = \sigma(w, \newbullet)$ and define $\tilde I$ analogously on $T^*Q$. Then $\liftmap_{\tilde \lambda}$ is the map
$$\xymatrix{ \liftmap_{\tilde \lambda}\colon  T_\lambda T^*M \ar[r]^-{I} & T^*_\lambda T^*M \ar[r]^{{\pi^2}^*} & T^*_{\tilde \lambda} T^*Q  \ar[r]^{\tilde I^{-1}} & T_{\tilde \lambda} T^*Q}.$$
This map gives us an Ehresmann connection on $\pi^2$.
\begin{lemma} \label{lemma:H2}
Define $\calV^2 = \ker \pi^2_*$ and let $\calH^2$ be its symplectic complement, i.e.,
$$\calH^2_{\tilde \lambda} = \left\{ \tilde w \in T_{\tilde \lambda} T^* Q \, \colon \, \tilde \sigma( \tilde w, \tilde u ) = 0 \text{ for any } \tilde u \in \calV^2_{\tilde \lambda} \right\}.$$

\begin{enumerate}[\rm (a)]
\item For any $\tilde \lambda \in T^*Q$, $\calH_{\tilde \lambda}^2$ is equal to the image of $\liftmap_{\tilde \lambda}$.
\item The map
$$\xymatrix{T(T^*Q) \ar[r]^-{p^Q_*} & TQ \ar[r]^{\pi_*}& TM}$$
restricted to $\calH^2$ is given by
$$\pi_* p^Q_*( \liftmap_{\tilde \lambda} w) = p^M_* w , \qquad \text{for any } \tilde \lambda \in T^*Q, w\in T_{\pi^2(\tilde \lambda)} T^*Q.$$
\item  $T(T^*Q) = \calH^2 \oplus \calV^2 $. Furthermore,
$$p^Q_* \calV^2 = \calV \quad \text{and} \quad p^Q_* \calH^2 = \calH.$$
\end{enumerate}
\end{lemma}
A subbundle of $T(T^*Q)$ is called \emph{symplectic} if $\sigma$ restricted to this subbundle is a non-degenerate two-form. Since $\calH^2$ is the symplectic complement of $\calV^2$, the result $T(T^*Q) = \calH^2 \oplus \calV^2$ is equivalent to~$\calV^2$, and hence also~$\calH^2$, being symplectic.

\begin{proof}
Recall that from the definition of the symplectic form, we have
\begin{equation}
\label{verticalSigma} \sigma(\vl_\lambda \alpha_0, w) = -\alpha_0(p^M_* w), \qquad \lambda, \alpha_0 \in T^*M, w \in T_\lambda T^*M.
\end{equation}

\begin{enumerate}[\rm (a)]
\item By definition $\tilde \sigma(\liftmap_{\tilde \lambda} w, \newbullet) = \sigma(w, \pi_*^2 \newbullet)$ vanishes on $\calV^2 = \ker \pi^2_*$. Since $\liftmap_{\tilde \lambda}$ is injective and its image is of rank equal to that of $\calH_{\tilde \lambda}^2$, it follows that it will also be surjective.
\item Observe that for any one-form $\alpha$ on $T^*M$, ${\pi^2}_* \vl \pi^* \alpha = \vl \alpha$ holds by definition of $\pi^2$. Hence, by \eqref{verticalSigma}, for any vector field~$W$ on~$T^*M$, we have
$$\alpha( p^M_* W) = \sigma(W, \vl \alpha) = \tilde \sigma(\liftmap W, \vl \pi^* \alpha) = \alpha(\pi_* p^Q_* \liftmap W).$$
Here, the vector field $\liftmap W$ on $T^*Q$ is defined by $SW |_{\tilde \lambda}= S_{\tilde \lambda} W |_{\pi^2(\tilde \lambda)}$. Since $\alpha$ and $W$ were arbitrary, the result follows.
\item Clearly, $p^Q_* \calV^2 = p^Q_* \ker {\pi^2}_* \subseteq \ker \pi_* = \calV$. Let $\zero$ denote the zero section of~$T^*Q$. Since $p_*^Q$ is a right inverse to the (injective) map $\zero_*: TQ \to T(T^*Q)$ and $\zero_*\calV \subseteq \calV^2$, we must have that $p^Q_*$ maps $\calV^2$ surjectively on $\calV$. To prove the analogous result for $\calH^2$, observe that for any $\tilde \lambda \in  T^*_q Q$ and $\tilde \alpha_0 \in \Ann(\calH)_q,$ we have $\vl_{\tilde \lambda} \tilde \alpha_0 \in \calV^2_{\tilde \lambda}$. Using \eqref{verticalSigma}, we know that any $\tilde w \in \calH_{\tilde \lambda}^2$ must satisfy
$$0 = -\tilde \sigma (\vl_{\tilde \lambda} \tilde \alpha_0, \tilde w) = \tilde \alpha_0(p^Q_* \tilde w) ,$$
implying that $p^Q_* \calH^2 \subseteq \calH$ since $\alpha_0 \in \Ann(\calH)$ was arbitrary. Finally, as $\pi_* p^Q_*$ is surjective on $TM$ and $\pi_*|_{\calH}$ is a linear isomorphism on every fiber, we must have that $p^Q_* \calH^2 = \calH$.

To complete the proof, we must show that $\calH^2$ and $\calV^2$ are indeed transverse. Assume that $\tilde w \in  (\calH^2 \cap \calV^2)|_{\tilde \lambda}.$ From our previous results, this intersection is contained in the kernel of $p^Q_*$. Hence, $\tilde w = \vl_{\tilde \lambda}\tilde \alpha_0$ for some $\tilde \alpha_0 \in T^*_{p^Q(\tilde \lambda)} Q$. However, since $\tilde \sigma(\vl_{\tilde \lambda} \alpha_0, \newbullet )$ must annihilate both $\calH^2$ and $\calV^2$ and since $\calH^2 \cup \calV^2$ is mapped surjectively onto $TM$, \eqref{verticalSigma} implies $\tilde \alpha_0 =0$. The result follows
\end{enumerate}
\end{proof}

\subsubsection{Horizontal lifts with respect to $\calH^2$} \label{sec:h2}
Lemma \ref{lemma:H2} (c) tells us that~$\calH^2$ is an Ehresmann connection on $\pi^2:T^*Q \to T^*M$. Let $h^2_{\tilde \lambda} w$ be the horizontal lift with respect to this connection. We want to see how this lift compares with the map $\liftmap_{\tilde \lambda}$.

\begin{lemma} \label{lemma:landh}
If
$$\xymatrix{ \tilde \lambda \in  T^*_q Q \ar@{|->}[r]^-{\pi^2} & \lambda \in T_{x}^* M} \quad \text{and} \quad  \xymatrix{ w \in T_\lambda T^*M  \ar@{|->}[r]^-{p^M_*}  & v \in T_x M }.$$
then the following relations hold
\begin{eqnarray} \label{ellandh1}
p_*^Q h^2_{\tilde \lambda} w & = &  p^Q_* \liftmap_{\tilde \lambda} w = h_q v, \\ \nonumber \\
\label{ellandh2} h^2_{\tilde \lambda} w &= & \liftmap_{\tilde \lambda} w - \vl_{\tilde \lambda}  \tilde \lambda \, \calR(h_q v, \newbullet ) .
\end{eqnarray}
As a consequence,
$${\pi^2}_* \liftmap_{\tilde \lambda} w = w + \vl_\lambda \tilde \lambda \, \calR(v, \newbullet ).$$
\end{lemma}
Written as a commutative diagram, 
Lemma \ref{lemma:landh} states that 
$$\begin{array}{m{3cm}m{3cm}m{3cm}}
\xymatrix{h^2_{\tilde \lambda} w \ar@{|->}[r]^-{{\pi^2}_*} \ar@{|->}[d]_{p^Q_*} & w \ar@{|->}[d]^{p^M_*}\\ h_q v \ar@{|->}[r]_-{\pi_*} & v} &
\begin{center} while \end{center} &
\xymatrix{\liftmap_{\tilde \lambda} w \ar@{|->}[r]^-{{\pi^2}_*} \ar@{|->}[d]_{p^Q_*} & w + \vl_\lambda \tilde \lambda \, \calR(v, \newbullet )\ar@{|->}[d]^{p^M_*}\\ h_q v \ar@{|->}[r]_-{\pi_*} & v} \end{array}.$$
Observe that $\vl_\lambda \tilde \lambda \, \calR(v, \newbullet )$ only depends on the projection of $\tilde \lambda$ to $\Ann(\calH)$. Hence, $h^2_{\tilde \lambda} w = \liftmap_{\tilde \lambda} w$ whenever $\tilde \lambda$ is in $\Ann(\calV)$.

\begin{proof}[Proof of Lemma \ref{lemma:landh}]
Equation \eqref{ellandh1} is immediate from the definition of horizontal lifts and Lemma \ref{lemma:H2} (b) and (c).

To prove \eqref{ellandh2}, let us split the Liouville form on $T^*Q$ into two parts $\tilde \theta = \theta^{\calH} + \theta^{\calV}$ where
$$\theta^{\calH}(w) = \tilde \lambda(\pr_{\calH} p^Q_* \tilde w), \qquad \theta^{\calV}(w) = \tilde \lambda(\pr_{\calV} p^Q_* \tilde w), \qquad \tilde w \in T_{\tilde \lambda} T^*Q.$$
We observe that for any $\tilde w \in T_{\tilde \lambda} T^*Q$,
\begin{align*} \theta^{\calH}(\tilde w)& = \tilde \lambda( \pr_{\calH} {p^Q}_* \tilde w) = \pi^2(\tilde \lambda)(\pi_* p^Q_* \tilde w) \\
& = \pi^2(\tilde \lambda)(p^M_* \pi^2_* \tilde w) = {\pi^2}^*(\theta)(\tilde w), \end{align*}
which in turn implies $\tilde \sigma = {\pi^2}^*\sigma - d\theta^{\calV}$. Define $ \calR^2$ as the curvature of $\calH^2$. We know from \eqref{ellandh1} that
$$p^Q_* \calR^2\left(h^2_{\tilde \lambda} w , h^2_{\tilde \lambda} u\right) = \calR\left(h_q v, h_q p^M_* u\right),$$
since for any pair of vector fields $W_j$, $j=1,2$, on $T^*M$ projecting to $X_j$ on $M$, we have
\begin{align*}
& p_*^Q \calR^2(h^2 W_1, h^2 W_2) = p_*^Q ([h^2 W_1, W_2] - h^2[W_1, W_2]) \\
& = [hX_1, hX_2]- h[X_1,X_2] = \calR(hX_1,hX_2)
\end{align*}
Here, we have used that $W_j$ is $p^M$-related to $X_j$ and $h^2 W_j$ is $p^Q$-related to $hX_j$.
Since $\calH^2$ is symplectic, $h^2_{\tilde \lambda} w$ is completely determined by the values of $\tilde \sigma(h^2_{\tilde \lambda} w, \newbullet)$ on vectors $h^2_{\tilde \lambda} u$, where $u \in T_\lambda TM$. We compute relations
\begin{align*}
& \tilde \sigma(h^2_{\tilde \lambda} w, h^2_{\tilde \lambda} u)  = \sigma(w,u) - d\theta^{\calV}(h^2_{\tilde \lambda} w, h_{\tilde \lambda}^2 u) \\
& = \sigma(w,u) + \theta(\calR^2( h^2_{\tilde \lambda} w, h_{\tilde \lambda}^2 u)) = \sigma(w,u) + \tilde \lambda \, \calR( h_qv, h_q p^M_* u) \\
& = \tilde \sigma(\liftmap_{\tilde \lambda} w, h^2_{\tilde \lambda} u) - \tilde \sigma \left(\vl_{\tilde \lambda} \tilde \lambda\, \calR( h_q v ,\newbullet) , h^2_{\tilde \lambda} u\right),
\end{align*}
which shows \eqref{ellandh2}
\end{proof}

We will state one of the observations made in the proof of Lemma \ref{lemma:landh} as a separate result, since we will need it later.
\begin{lemma} \label{lemma:symplecticsplit}
Let $\theta^{\calV}$ be the one-form $\theta^{\calV} |_{\tilde \lambda}= \tilde \lambda(\pr_{\calV} p^Q_*  \newbullet)$ on $T^*Q$. Then
$$\tilde \sigma = {\pi^2}^* \sigma - d \theta^{\calV}.$$
\end{lemma}

We next turn to the key lemma which will give us our proof of Theorem~\ref{th:main}.
\begin{lemma} \label{lemma:horver}
Let $\lambda(t)$ be a curve in $T^*M$ with $p^M( \lambda(t)) = \gamma(t)$. Let $\tilde \lambda(t)$ be an $\calH^2$-horizontal lift of $\lambda(t)$ with $p^Q (\tilde \lambda(t)) = \tilde \gamma(t)$. Then $\tilde \gamma(t)$ is an $\calH$-horizontal lift of $\gamma(t)$. Furthermore, $\tilde \lambda(t) = \lambda^{\calH}(t) + \lambda^{\calV}(t)$, where
$\lambda^{\calH}$ is the curve in $\Ann(\calV)$ determined by
\begin{equation} \label{lambdahor} \lambda^{\calH}(t)(\tilde v) = \lambda(t)(\pi_* \tilde v) \text{ for any } \tilde v \in T_{\tilde \gamma(t)}Q,\end{equation}
and $\lambda^{\calV}$ is a curve in $\Ann(\calH)$ satisfying $\bnabla_{\dot \gamma(t)} \lambda^{\calV}(t) = 0.$
\end{lemma}

\begin{proof}
The curve $\tilde \gamma(t)$ is an $\calH$-horizontal lift of $\gamma(t)$ by \eqref{ellandh1}. Let $\lambda^{\calH}(t)$ be defined as in \eqref{lambdahor} with $\tilde \gamma(t)$ a $\calH$-horizontal lift of $\gamma(t)$. Since $\lambda^{\calH}(t) = \pi_{\tilde \gamma(t)}^* \lambda(t)$, we know that $\pi^2(\lambda^{\calH}(t)) = \lambda(t)$ , so all we need to show is that $\lambda^{\calH}(t)$ is $\calH^2$-horizontal to prove that it is a lift of $\lambda(t)$. For a sufficiently short segment of $\gamma(t)$, let $X$ be a vector field on $M$ such that
\begin{equation} \label{Xgamma} X|_{\gamma(t)} = \dot \gamma(t). \end{equation}
Let $W^{\calH}$ be a local vector field on $T^*Q$ such that $p^Q_* W^{\calH} = hX$ and such that $W^{\calH} |_{\lambda^{\calH}(t)} = \dot \lambda^{\calH}(t)$ for a sufficiently short segment. Let $\tilde W$ be an arbitrary vector field on $T^*Q$ with values in $\calV^2 = \ker \pi^2_*$. Then by Lemma~\ref{lemma:symplecticsplit} and Cartan's formula,
\begin{align*} & \sigma(\dot \lambda^{\calH}(t), \tilde W) = - d\theta^{\calV}|_{\lambda^{\calH}(t)} (W^{\calH}, \tilde W) \\
& = \tilde W \theta^{\calV}(W^{\calH}) |_{\lambda^{\calH}(t)}- W^{\calH} \theta^{\calV}(\tilde W)|_{\lambda^{\calH}(t)},
\end{align*}
since $\theta^{\calV}|_{\lambda^{\calH}(t)} = 0$. However, $\theta^{\calV}|_{\tilde \lambda}(W^{\calH}) = \tilde \lambda(\pr_{\calV} hX) = 0$ and
$$ W^{\calH} \theta^{\calV}(\tilde W)|_{\lambda^{\calH}(t)} = \frac{d}{dt} \lambda^{\calH}(t) \pr_{\calV} p_*^Q \tilde W|_{\lambda^{\calH}(t)} = 0,$$
so in conclusion $\sigma(\dot \lambda^{\calH}(t), \tilde W_2)= 0$ and $\lambda^{\calH}(t)$ is hence tangent to $\calH^2$.
 
Since $\pi^2$ commutes with addition on respectively $T^*Q$ and $T^*M$, both $\calV^2$ and $\calH^2$ are preserved under addition. In particular, the sum of two $\calH^2$-horizontal curves are again $\calH^2$-horizontal.
As a consequence, if $\tilde \lambda(t)$ is any horizontal lift of~$\lambda(t)$, then $\lambda^{\calV}(t) = \tilde \lambda(t) - \lambda^{\calH}(t)$ must be an $\calH^2$-horizontal lift of $\zero|_{\tilde \gamma(t)}$, the zero section along $\tilde \gamma(t)$. In particular $\lambda^{\calV}(t)$ is a curve in $\Ann(\calH)$.

Again for a sufficiently short segment of $\lambda^{\calV}(t)$, let $W^{\calV}$ be a local vector field such that $W^{\calV}|_{\lambda^{\calV}(t)} = \dot \lambda^{\calV}(t)$ and such that $p^Q_* W^{\calV} = hX$ where $X$ is a vector field such as in \eqref{Xgamma}. Let $V$ be any vector field on $Q$ with values in $\calV$ and let $\tilde W$, be any vector field in $\calV^2$ satisfying $p^Q_* \tilde W = V.$ Then since $\lambda^{\calV}(t)$ is a $\calH^2$-horizontal lift and since $\theta^{\calV}(W^{\calV} |_{\tilde \lambda}) = \tilde \lambda \pr_{\calV} hX = 0$ for any $\tilde \lambda \in T^*Q$, Lemma~\ref{lemma:symplecticsplit} and Cartan's formula tells us that
\begin{align} \label{getbnabla}
0 & = - \sigma(\dot \lambda^{\calV}(t), \tilde W) =  d\theta^{\calV}|_{\lambda^{\calV}(t)} (W^{\calV}, \tilde W) \\ \nonumber
& =W^{\calV} \theta(\tilde W)|_{\lambda^{\calV}(t)} - \theta^{\calV}|_{\lambda^{\calV}(t)} ([W^{\calV}, \tilde W]) \\ \nonumber
& =\frac{d}{dt} \lambda^{\calV}(t) V - \lambda^{\calV}(t) ([hX, V]) = (\bnabla_{\dot \gamma} \lambda^{\calV}(t) )(V).
\end{align}
In conclusion, $\bnabla_{\dot \gamma} \lambda^{\calV}(t) = 0$.
 \end{proof}

\subsubsection{Proof of Theorem \ref{th:main}} \label{sec:proofmainsec}
Let $H$ be any smooth function on $T^*M$ and let $\tilde H = H \circ \pi^2$. By the definition of $\liftmap$ and \eqref{prelift}, $\liftmap \vec{H}$ is the Hamiltonian vector field of $\tilde H$.

To prove (a), observe that $\liftmap \vec{H}$ and $h^2 \vec{H}$ coincide on $\Ann(\calV)$ by Lemma~\ref{lemma:landh}. Hence, if $\lambda(t)$ is an integral curve of $H$ projecting to $\gamma(t)$, then $\lambda^{\calH}(t)$ defined as in~\eqref{lambdahor} with respect to some horizontal lift $\tilde \gamma(t)$ of $\gamma(t)$ is an integral curve of $\tilde H$. Since all elements in $\Ann(\calV)$ are pull-backs of elements in $T^*M$, it follows that all integral curves of $\tilde H$ with initial condition in $\Ann(\tilde \calV)$ will be of this form. 

For the proof of (b), introduce the map $\Phi^{\calV}:T^*Q \to T^*Q$ defined by $\Phi^{\calV}(\tilde \lambda) = \pr_{\calV}^* \tilde \lambda$. We again use Lemma~\ref{lemma:landh}~(b), 
$$ \dot \lambda(t) = \pi^2_* \liftmap \vec{H}|_{\tilde \lambda(t)} = \vec{H}|_{\lambda(t)} + \vl_{\lambda(t)} \tilde \lambda(t) \calR( \dot \gamma(t), \newbullet )
 = \vec{H}|_{\lambda(t)} + \vl_{\lambda(t)} \beta(t)\calR( \dot \gamma(t), \newbullet) ,
$$
where $\beta (t) = \Phi^\calV(\lambda(t))$. If $p^Q(\tilde \lambda) = q$, the equality
\begin{align*} &- {\Phi^{\calV}}_*\left(h^2_{\tilde \lambda} w - \liftmap_{\tilde \lambda} w\right) = {\Phi^{\calV}}_* (\vl_{\tilde \lambda} \tilde \lambda \calR(h_q p^M_* w, \newbullet)) \\
& = \left. \frac{d}{ds} \pr_{\calV}^*(\tilde \lambda + s\tilde \lambda \calR(h_q p^M_* w, \newbullet)) \right|_{s=0} = 0,\end{align*}
and Lemma~\ref{lemma:horver} give us $\bnabla_{\dot \gamma} \beta(t) = 0$, where $\gamma(t)$ is the projection of $\lambda(t)$ to $M$. Finally, we know that $\tilde \gamma(t) = p^Q(\tilde \lambda(t))$ is an $\calH$-horizontal lift of $ \gamma(t) = p^M(\lambda(t))$ from Lemma~\ref{lemma:landh}~(a).

\subsubsection{Representation in local coordinates} \label{sec:localcoordinates}
We will use the coordinate of \eqref{coordinates}. The bundle $\calV^2 = \ker \pi^2_*$ is spanned by $\partial_{y_1}, \dots, \partial_{y_\nu}$ and $\partial_{b_1}, \dots, \partial_{b_\nu}$. 
Relative to the same coordinates, $\calH^2$ is spanned by
$$h^2 \partial_{x_j} = h\partial_{x_j} + \sum_{\kappa,\mu=1}^\nu b_{\mu} {\bf \Gamma}_{j\kappa}^\mu \partial_{b_\kappa} \qquad \text{ and } \qquad h^2 \partial_{a_j} = \partial_{a_j},$$
or
$$\liftmap \partial_{x_j} = h \partial_{x_j} -\sum_{\kappa =1}^\nu \sum_{i=1}^n b_{\kappa} \calR_{ij}^\kappa \partial_{a_i} + \sum_{\kappa, \mu=1}^\nu b_\mu {\bf \Gamma}_{j\kappa}^\mu \partial_{b_\kappa} \qquad \text{and} \qquad \liftmap \partial_{a_j} = \partial_{a_j}  ,$$
$$\calR_{ij}^\kappa = \tau_k ([h\partial_{x_i}, h\partial_{x_j}]) , \qquad {\bf \Gamma}_{i\mu}^\kappa = \tau_\kappa ([h\partial_{x_i}, \partial_{y_\mu}]) .$$

\subsection{Proof of Proposition~\ref{prop:sRabnormal}} \label{app:abnormal}
We will use the notation of Section \ref{sec:proofmain}. Let~$\tilde \lambda(t)$ be a characteristic of $\tilde D$ and let $\lambda(t)$, $\beta(t)$, $\tilde \gamma(t)$ and $\gamma(t)$ be defined as in Proposition~\ref{prop:sRabnormal}. We make the following observations.
\begin{enumerate}[\rm (i)]
\item If $\lambda, \alpha_0 \in \Ann(\tilde D)_q$, $q \in Q$, then $\vl_{\lambda} \alpha_0 \in T_\lambda \Ann(\tilde D)$.
\item If $\alpha(t)$ is a curve in $\Ann(D)$ and $\tilde \alpha(t)$ is its $\calH^2$-horizontal lift, then by Lemma~\ref{lemma:horver}, $\tilde \alpha(t)$ is a curve in $\Ann(\tilde D)$. Hence, for any $\tilde \lambda \in \Ann(\tilde D)$ and $w \in T_{\pi^2(\tilde \lambda)} \Ann(D)$, we have $h_{\tilde \lambda}^2 w \in T_{\tilde \lambda}\Ann(\tilde D)$.
\item If $\tilde \alpha$ is a one-form vanishing on $\tilde D$, it can be seen as a map $\tilde \alpha: Q \to \Ann(\tilde D)$, and so the image of $\tilde \alpha_* $ is contained in $T\Ann(\tilde D)$.
\end{enumerate}
Since $\tilde \lambda(t)$ is a characteristic, $\sigma(\dot {\tilde \lambda}(t), \newbullet)$ vanishes on all vectors of the type (i), (ii) and (iii). Looking at each of these requirements will give us respectively (i), (ii) and (iii) of Proposition~\ref{prop:sRabnormal}. 

To start with type (i), using \eqref{verticalSigma} we obtain that for any $\tilde \alpha \in \Gamma(\Ann(\tilde D)),$
$$0 = \sigma(\dot{\tilde \lambda}(t), \vl \tilde \alpha) =  \tilde \alpha(\dot {\tilde \gamma}(t)),$$
meaning that $\tilde \gamma(t)$ has to be $\tilde D$-horizontal. Since $\pi_* \tilde D = D$, we know that $\gamma(t)$ is $D$-horizontal.

Continuing with type (ii), let $\alpha \in \Gamma(\Ann(D))$ be a one-form on $M$ with $\alpha|_{\gamma(t)} = \lambda(t)$. Then $\alpha_*(TM) \subseteq T\Ann(D)$, and so for any $v \in T_{\gamma(t)} M$,
\begin{align*}
0 & = \tilde \sigma(\dot {\tilde \lambda}(t), h_{\tilde \lambda(t)}^2 \alpha_* v) = \tilde \sigma\left(\dot {\tilde \lambda}(t), S_{\tilde \lambda(t)} \alpha_* v - \vl_{\tilde \lambda(t)} \tilde \lambda(t) \calR(h_{\tilde \gamma(t)} v, \newbullet)\right) \\
& = \sigma(\dot \lambda(t), \alpha_* v) + \tilde \lambda(t) \calR(\dot \gamma(t), v),
\end{align*}
 by Lemma~\ref{lemma:landh}. Write $p^{\Ann(D)}: \Ann(D) \to M$ for the natural projection. Since $p_*^{\Ann(D)} \alpha_*$ is equal to the identity on $TM$, the image of $\alpha_*$ in $T_{\lambda(t)} \Ann(D)$ must be transverse to $\ker p^{\Ann(D)}_*$. Hence, any $w \in T_{\lambda(t)} \Ann(D)$ is on the form $w = \alpha_* v + \vl_{\lambda(t)} \eta_0$ for some $v \in T_{\gamma(t)} M$, $\eta_0 \in \Ann(D)_{\gamma(t)}$. This lets us write
\begin{align*}
& \sigma(\dot \lambda(t), w) = \sigma(\dot \lambda(t), \alpha_* v + \vl_{\lambda(t)} \eta_0)   = - \tilde \lambda(t) \calR(\dot \gamma(t), v) - \eta_0(\dot \gamma(t)) \\
& = - \beta(t) \calR(\dot \gamma(t), v) = - \beta(t) \calR(\dot \gamma(t), p_*^M w).
\end{align*}

To determine the equations of $\beta(t)$, we consider vectors of type (iii). Let $V$ be any vector field on $Q$ with values in $\calV$. For one-form $\tilde \alpha\in \Gamma(\Ann(\tilde D))$ with $\tilde \alpha|_{\tilde \gamma(t)} = {\tilde \lambda}(t)$, we get
\begin{align*}
& 0 = \tilde \sigma(\dot {\tilde \lambda}(t) , \tilde \alpha_* V) = \sigma(\dot \lambda(t), \pi^2_* \tilde \alpha_* V) - d\theta^{\calV}(\dot {\tilde \lambda}(t) , \tilde \alpha_* V) \\
& = - \beta(t) \calR( \dot \gamma(t), p_*^M \pi^2_* \tilde \alpha_* V)  - d\theta^{\calV}(\dot {\tilde \lambda}(t) , \tilde \alpha_* V) .
\end{align*}
Since $p^M \circ \pi^2 \circ \alpha = \pi$, the first term vanishes. The last term is equal to $-(\bnabla_{\dot \gamma} \beta(t) )(V) $ by the argument in \eqref{getbnabla}.

\bibliographystyle{habbrv}
\bibliography{Bibliography}

\end{document}